\documentclass[12pt]{article}

\usepackage{amssymb}
\usepackage{bm}                
\usepackage{mathrsfs}      
\usepackage{amsmath, amsthm}
\usepackage{graphicx,epsf,subfigure,pstricks,pst-node,psfrag,amsthm,amssymb,amsmath}
\usepackage{float}
\usepackage{array}

\usepackage{natbib}

\newcommand{\Ind}{1\!\mathrm{l}}
\newcommand{\iy}{\infty}
\renewcommand{\th}{\theta}
\numberwithin{equation}{section}
\usepackage[margin=.8in]{geometry}

\begin{document}

\title{Adaptive Bayesian procedures using random series priors}
\author{Weining Shen \and Subhashis Ghosal}
\date{}
\maketitle

\begin{center}
\textbf{Abstract}
\end{center}
We consider a general class of prior distributions for nonparametric Bayesian estimation which uses finite random series with a random number of terms. A prior is constructed through distributions on the number of basis functions and the associated coefficients. We derive a general result on adaptive posterior contraction rates for all smoothness levels of the target  function in the true model by constructing an appropriate ``sieve'' and applying the general theory of posterior contraction rates. We apply this general result on several statistical problems such as density estimation, various nonparametric regressions, classification, spectral density estimation, functional regression etc. The prior can be viewed as an alternative to the commonly used Gaussian process prior, but properties of the posterior distribution can be analyzed by relatively simpler techniques. An interesting approximation property of B-spline basis expansion established in this paper allows a canonical choice of prior on coefficients in a random series and allows a simple computational approach without using Markov chain Monte-Carlo (MCMC) methods. A simulation study is conducted to show that the accuracy of the Bayesian estimators based on the random series prior and the Gaussian process prior are comparable. We apply the method on Tecator data using functional regression models.

\vspace*{.1in}


\newpage

\newtheorem{theorem}{Theorem}
\newtheorem{lemma}{Lemma}
\newtheorem{remark}{Remark}
\newtheorem{definition}{Definition}
\newtheorem{example}{Example}
\newtheorem{corollary}{Corollary}
\newtheorem{ex}{}
\newtheorem{proposition}{Proposition}

\section{Introduction}\label{sec1}
Bayesian methods have been widely used in the nonparametric statistical literature. Contraction rates of posterior distributions were studied in \citet{Ghosal2000}, \citet{Shen2001}, \citet{Ghosal20071,Ghosal2007} and \citet{Vander2008}. The optimal contraction rate of estimating a univariate $\alpha$-smooth function is typically  $n^{-\alpha/(2\alpha+1)}$, where $n$ is the sample size. Since the smoothness parameter $\alpha$ is usually unknown in practice, it is then of interest to investigate if a prior leads to optimal posterior contraction rates simultaneously for all values of $\alpha$, possibly up to a logarithmic factor. If that holds, a procedure is called rate-adaptive.

Bayesian rate adaptation results are important for at least two reasons. First, they guarantee maximum possible accuracy of the Bayesian estimation procedure within the given framework. Secondly, they assure that the same prior distribution can be used regardless of the smoothness of the underlying function being estimated. Bayesian adaptation results have been established for signal estimation by \citet{Belitser2003} and \citet{szabo2013}, for density estimation by \citet{Ghosal2002, Ghosal2008}, \citet{Scricciolo2006}, and \citet{Huang2004}, and for nonparametric regression by \citet{Huang2004} using discrete mixtures. Alternatively, \citet{Vander2009} constructed a prior based on a randomly rescaled Gaussian process, which automatically adapts for a continuous range of smoothness parameters.

Gaussian processes have been widely used for constructing prior distributions \citep{Lenk1988} and applications in spatial statistics \citep{Banerjee2008}. Posterior computational methods were developed in \citet{Choudhuri2007}, \citet{Rasmussen2006}, \citet{Tokdar2007} and \citet{Rue2009} among others. Posterior asymptotic properties, which are primarily driven by the structure of their reproducing kernel Hilbert space, were studied by \citet{Tokdar20071}, \citet{Ghosal2006}, \cite{Choi2007}, \citet{Vander2007,Vander2008,Vander2009}, \citet{Castillo2008,Castillo2012},  \citet{Castillo2014} and \citet{Bha2014}.

Besides a Gaussian process, another common prior on functions, obtained by putting a prior on the the number of terms and the corresponding coefficients of a series expansion, has been used extensively in applications \citep{Crainiceanu2005}. Study of posterior contraction rates for such finite random series priors have begun only recently. \citet{Rivoirard2012} considered univariate density estimation using an exponential link and wavelet or Fourier series basis;  \citet{dejonge2012} considered a general approach for multivariate function estimations using tensor-product spline basis and Gaussian distributions on the coefficients; \citet{Arbel2013}  proposed a class of sieve priors with general choice of basis functions and independent priors on the coefficients. A related work is \citet{Babenko2010}, who obtained oracle inequalities for posterior contraction for the infinite dimensional normal mean problem by putting a prior on the number of non-zero entries and then independent normal priors on the resulting components.

In the present paper, our contributions are two-fold. First, we obtain posterior contraction rates for finite random series priors for any curve estimation problems under both univariate and multivariate settings with arbitrary bases and arbitrary distributions on coefficients. Second, we show that for the B-splines basis and certain choices of priors on the coefficients, the posterior computation can be carried out by exploiting a conjugacy-like structure without using Markov chain Monte-Carlo (MCMC) techniques. Inevitably there are some overlap with \citet{Rivoirard2012}, \citet{dejonge2012} and \citet{Arbel2013}, but our goal is to emphasize the general properties of finite random series in all curve estimation problems and that the availability of conjugacy-like structures, which emerges only when one considers general prior distributions on the coefficients. We formulate one general theorem in an abstract setting suitable as a prelude for many different inference problems where we allow arbitrary basis functions and arbitrary multivariate distributions on the coefficients of the expansion. Thus the resulting process induced on the function need not be Gaussian, and can accommodate a variety of functions starting from one with a bounded support to one with a heavy tail. The resulting rate obtained in the abstract theorem depends on the smoothness of the underlying function, approximation ability of the basis expansion used, tail of the prior distribution on the coefficients, prior on the number of terms in the series expansion, prior concentration and the metrics being used. We compute the rates for various combinations of these choices.

It may be noted that Gaussian process and random series priors are intimately related in two ways --- a normal prior on the coefficients of a random series gives a Gaussian process while the Karhunen-Lo\`{e}ve expansion of a Gaussian process expresses itself as a random series with basis consisting of eigenfunctions of the covariance kernel of the Gaussian process.
Thus a random series prior may be regarded as a  flexible alternative to a Gaussian process prior. It is interesting to note that the theory of posterior contraction for Gaussian process priors established in \citet{Vander2007,Vander2008,Vander2009} use deep properties of Gaussian processes, while relatively elementary techniques lead to comparable posterior contraction rates for finite random series priors. Posterior computation for Gaussian process priors often need reversible jump MCMC procedures \citep{Tokdar2007} typically with a large number of knots  to approximate a given Gaussian process. For a random series prior based on B-spline expansion, for an appropriate prior on the coefficients, the conjugacy-like structure model can avoid the use of MCMC altogether by representing the posterior mean analytically, although the number of terms in the representation may be large. When the sample size $n$ is relatively small (e.g. $n=10$), the number of terms is manageable and the exact values of posterior moments can be computed. When the sample size is large, we sample a few terms and estimate the sum. The Monte Carlo standard error of the expression can be estimated, and is often fairly controlled provided the terms are similar to each other.

The paper is organized as follows. In Section \ref{sec3}, we present the main theorems of random series priors. In Sections \ref{sec5} and \ref{sec:reg}, we apply the theorems to a variety of statistical problems and derive the corresponding posterior contraction rates. Numerical results are presented in Section \ref{sec9}.

\section{General results}\label{sec3}
\subsection{Notations}
Let $\mathbb{N}= \{1,2,\ldots\}$, $\Delta_j = \{(x_1,\ldots,x_j): \sum_{i=1}^j x_i =1, x_1,\ldots,x_j \geq 0 \}$, and $\delta_x$ stand for the degenerate probability distribution at a point $x$. Let the indicator function of a set $A$ be denoted by $\Ind\{A\}$.
For an open region $\Omega_0$ in a Euclidean space, define the $\alpha$-H\"{o}lder class $\mathcal{C}^{\alpha}(\Omega_0)$ as the collection of functions $f$ on $\Omega_0$ that has bounded derivatives up to the order $\alpha_0$, which is the largest integer strictly smaller than $\alpha$, and the $\alpha_0$-th derivative of $f$ satisfies the H\"{o}lder condition
  $  |f^{(\alpha_0)} (x) - f^{(\alpha_0)}(y) | \leq C |x-y|^{\alpha-\alpha_0}$
  for some constant $C>0$ and any $x,y$ in the support of $f$.

We use ``$\lesssim$'' to denote an inequality up to a constant multiple, where the underlying constant of proportionality is universal. By $f \asymp g$, we mean $f \lesssim g \lesssim f$. The packing number $D(\epsilon,T,d)$ is defined as the maximum cardinality of a subset of $T$ whose elements are at least $\epsilon$-separated out with respect to a distance $d$. Let $h^2(p,q) = \int (\sqrt{p} - \sqrt{q} )^2 d\mu$, be the squared Hellinger distance, $K(p,q) = \int p \log (p/q) d\mu$, $V(p,q) = \int p  \log^2 (p/q)  d\mu$, be the Kullback-Leibler (KL) divergences and $\mathcal{K}(p,\epsilon) = \{f: K(p,f) \leq \epsilon^2, V(p,f) \leq \epsilon^2 \}$, be the KL neighborhood.
For a vector $\bm{\theta} \in \mathbb{R}^d$, define $\|\bm{\theta} \|_{p} = \{ \sum_{i=1}^d |\theta_i|^p \}^{1/p} $, $1\le p<\infty$, and $\|\bm{\theta} \|_{\infty} = \max_{1\leq i \leq d} |\theta_i|$. Similarly, we define $\| f \|_{p,G} = \{ \int |f(x)|^p dG \}^{1/p}$ and $\|f \|_{\infty} = \sup_x |f(x)|$ as the $L_p$-, $1 \leq p < \infty$, and $L_{\infty}$-norms of a function $f$ with respect to a measure $G$.

\subsection{Main results}\label{sec31}
We consider a random variable $J$ taking values in $\mathbb{N}$. For each $J \in \mathbb{N}$, we consider a triangular array of linearly independent real-valued functions $\bm{\xi} = (\xi_1,\xi_2,\ldots,\xi_{J})^T$ defined on a region $\Omega_0$. In applications, $\Omega_0$ will be typically a bounded region. Note that the resulting basis functions may change from one stage to the next, although that is not made explicit in our notation. We use $\Pi$ as a generic notation for priors assigned on $J$ and the coefficients of basis functions $\bm{\theta}=(\theta_1,\ldots,\theta_{J})^T$.
\begin{enumerate}
\item[(A1)] For some $c_1,c_2 >0$, $0 \leq t_2 \leq t_1 \leq 1$,
$\exp\{ - c_1 j \log ^{t_1}j \}  \leq \Pi (J = j) \le \exp\{-c_2 j \log^{t_2}j \}.$
\item[(A2)] Given $J$, we consider a $J$-dimensional joint distribution as the prior for $\bm{\theta}=(\theta_1,\ldots,\theta_{J})^T$ satisfying $\Pi(\| \bm{\theta} -\bm{\theta}_0 \|_2 \leq \epsilon ) \geq \exp \{ - c_3 J \log (1/\epsilon) \}$ for every $\|\bm{\theta_0}\|_{\infty} \leq H$, where $c_3$ is some positive constant, $H$ is chosen sufficiently large and $\epsilon >0$ is sufficiently small. Also, assume that $\Pi(\bm{\theta} \notin [-M,M]^J) \leq J \exp\{- C M^{t_3}\}$ for some constants $C, t_3 > 0$.
\end{enumerate}

\begin{remark}\rm
Geometric, Poisson and negative binomial distributions on $J$ satisfy Condition (A1) respectively with $t_1=t_2=0$, $t_1=t_2=1$ and $t_1=t_2=0$. Examples of priors satisfying (A2) include independent gamma, exponential distributions assigned on each element of $\bm{\theta}$ and multivariate normal and Dirichlet distributions provided the parameters lie in a fixed compact set; see Lemma 6.1 of \citet{Ghosal2000} for the last conclusion.
\end{remark}

We consider a distance metric $d$ on functions belonging to $\Omega_0$ satisfying the following condition for every $\bm{\theta_1},\bm{\theta}_2\in \mathbb{R}^J$, $J \in \mathbb{N}$, and some positive increasing function $a(\cdot)$:
\begin{eqnarray}
d(\bm{\theta}_1^T \bm{\xi} ,\bm{\theta}_2^T \bm{\xi} )  \leq   a(J)\|\bm{\theta_1} -\bm{\theta}_2 \|_2.  \label{eq:e302}
\end{eqnarray}

Now we state the main theorem, which gives unified conditions for posterior contraction rates for various inference problems, in a manner similar to Theorem 2.1 of \citet{Vander2008} and Theorem 3.1 of \citet{Vander2009}.

\begin{theorem}\label{thm:t0}
Let $\epsilon_n \geq \bar{\epsilon}_n$ be two sequence of positive numbers satisfying $\epsilon_n \rightarrow 0$ and $n \bar{\epsilon}_n^2 \rightarrow \infty$ as $n \rightarrow \infty$. For a function $w_0$, suppose that a prior satisfies Conditions (A1) and (A2). Assume that there exist sequences of positive numbers $J_n$, $\bar{J}_n$ and $M_n$, a strictly decreasing, nonnegative function $e(\cdot)$ and  $\bm{\theta}_{0,j} \in \mathbb{R}^j$ for any $j \in \mathbb{N}$, such that the following conditions hold for some positive constants $a_1>1$, $a_2$, $C_0$ and $H$:
\begin{eqnarray}
& \|\bm{\theta}_{0,j}\|_{\infty} \leq H,~~ d(w_0,\bm{\theta}_{0,j}^T \bm{\xi}) \leq e(j), \label{eq:e301n} \\
& J_n \left\{\log J_n + \log a(J_n) + \log M_n + C_0 \log n \right\} \leq n\epsilon_n^2, \label{eq:e302n} \\
& e(\bar{J}_n) \leq \bar{\epsilon}_n,~~c_1 \bar{J}_n \log^{t_1} \bar{J}_n  + c_3 \bar{J}_n \log (2a(\bar{J}_n) /\bar{\epsilon}_n) \leq a_2 n \bar{\epsilon}_n^2,
\label{eq:e303n} \\
&  n\bar{\epsilon}_n^2 \leq C J_n \log^{t_2} J_n~ \text{for any constant}~C,~~J_n \exp\{-C M_n^{t_3}\} \leq (a_1- 1) \exp \{- n\bar{\epsilon}_n^2\}. \label{eq:e304n}
\end{eqnarray}
Let $\mathcal{W}_{J_n,M_n} = \{w = \bm{\theta}^T \bm{\xi}: \bm{\theta} \in \mathbb{R}^j,  j\leq J_n, \|\bm{\theta}\|_{\infty} \leq M_n \}$. Then the following assertions hold:
\begin{eqnarray}
\log D(n^{-C_0},\mathcal{W}_{J_n,M_n},d) & \leq & n \epsilon_n^2, \label{eq:e309} \\
 \Pi (W \notin \mathcal{W}_{J_n,M_n} ) & \leq & a_1 \exp \{-bn \bar{\epsilon}_n^2\}, \label{eq:e310} \\
-\log \Pi \{ w = \bm{\theta}^T \bm{\xi} :d( w_0,w ) \leq  \bar{\epsilon}_n \} & \leq & a_2 n\bar{\epsilon}_n^2. \label{eq:e311}
\end{eqnarray}
\end{theorem}

\begin{proof}
We first verify \eqref{eq:e309}, using the definition of packing number, the assumptions on $M_n$, $J_n \geq 2$, the fact that $a(\cdot)$ is increasing and \eqref{eq:e302}, we obtain
\begin{eqnarray}
 \lefteqn{ \log D(n^{-C_0},\mathcal{W}_{J_n,M_n},d) } \nonumber \\
& \leq & \log \Big\{ \sum_{j=1}^{J_n} D(n^{-C_0} / a(j), \{ \bm{\theta} \in  \mathbb{R}^j, \|\bm{\theta}\|_{\infty} \leq M_n\}, \| \cdot \|_2) \Big\} \nonumber \\
& \leq & \log \Big[ J_n \Big\{ \sqrt{J_n} M_n a(J_n) n^{C_0} \Big\}^{J_n} \Big] \nonumber \\
& \leq & J_n (\log J_n + \log M_n + \log a(J_n) + C_0 \log n ) \leq n \epsilon_n^2.
\end{eqnarray}
Next, to verify \eqref{eq:e310}, observe that for some $c_2' > 0$, 
\begin{eqnarray}
\Pi (w \notin \mathcal{W}_{J_n,M_n} ) & \leq & \Pi (J > J_n) + \sum_{j=1}^{J_n} \Pi  ( \bm{\theta} \notin [-M_n,M_n]^j ) \Pi (J=j) \nonumber \\
                                     & \leq & \exp(-c_2' J_n \log^{t_2} J_n ) + J_n \exp \{-C M_n^{t_3}\} \nonumber \\
                                     & \leq &  a_1 \exp \{- n\bar{\epsilon}_n^2\}.
\end{eqnarray}
For \eqref{eq:e311}, using \eqref{eq:e301n}, since $d(w_0,\bm{\theta}_{0,j}^T \bm{\xi}) \leq e(j) \leq \bar{\epsilon}_n$ for all $j \geq \bar{J}_n$, we have
\begin{eqnarray}
\Pi \{w:d(w_0 , \bm{\theta}^T \bm{\xi}) \leq 2\bar{\epsilon}_n \} & \geq & \Pi(J = \bar{J}_n) \Pi \left(\| \bm{\theta} -\bm{\theta}_0 \|_2 \leq  \bar{\epsilon}_n/a(\bar{J}_n) \right)  \nonumber \\
 & \geq & \exp\{ - c_1 \bar{J}_n \log^{t_1} \bar{J}_n\} \exp \Big\{ - c_3 \bar{J}_n \log \Big(\frac{a(\bar{J}_n)}{  \bar{\epsilon}_n}\Big) \Big\}.
\end{eqnarray}
By taking the negative of the logarithm on both sides, and using \eqref{eq:e303n}, we obtain \eqref{eq:e311}.
\end{proof}

Conditions \eqref{eq:e302n}--\eqref{eq:e304n} require sufficiently large $J_n, \bar{J}_n$ in order to have sufficiently good approximation to $w_0$ while $J_n, \bar{J}_n$ should not be too large if the complexity of the model is to be controlled. When studying Bayesian asymptotic properties, a balance between bias and complexity needs to be established to obtain the optimal posterior contraction rate.

Theorem \ref{thm:t0} can be further simplified to obtain a posterior contraction rate at $w_0$. We assume that the approximation error is of the form $e(J) \lesssim J^{-\alpha}$ for $\alpha$-smooth functions. Such collections include B-splines, wavelets, Fourier series and many other commonly used bases. Let $d$ be the $L_2$-distance or the $L_{\infty}$-distance. For two groups of densities $p_{i,w_1},p_{i,w_2},i=1,\ldots,n$, we consider the root average squared Hellinger distance, defined by $\rho_n^2(w_1,w_2) = n^{-1} \sum_{i=1}^n h^2(p_{i,w_1},p_{i,w_2})$. Note that when the observations are i.i.d., $\rho_n$ reduces to the usual Hellinger distance. Then the following result gives the posterior contraction rate for various inference problems.

\begin{theorem}\label{thm:t1}
Suppose that we have independent observations $X_i$ following some distributions with densities $p_{i,w}$, $i=1,\ldots,n$ respectively. Let $w_0 \in \mathcal{C}^{\alpha}(\Omega_0)$ be the true value of $w$. Let $r$ be either $2$ or $\infty$. Let $\epsilon_n \geq \bar{\epsilon}_n$ be two sequence of positive numbers satisfying $\epsilon_n \rightarrow 0$ and $n \bar{\epsilon}_n^2 \rightarrow \infty$ as $n \rightarrow \infty$. Assume that there exists a $\bm{\theta_0} \in \mathbb{R}^J$, $\|\bm{\theta}_0\|_{\infty} \leq H$ and some positive constants $C_1,C_2$ and $K_0\ge 0$ satisfying
\begin{align}
\|w_0 -\bm{\theta}_0^T \bm{\xi}\|_r & \leq C_1 J^{-\alpha},  \label{eq:eth1}\\
\|\bm{\theta}_1^T\bm{\xi} -\bm{\theta}_2^T\bm{\xi}\|_r & \leq C_2 J^{K_0} \|\bm{\theta}_1-\bm{\theta}_2\|_2,~\bm{\theta}_1,\bm{\theta}_2 \in \mathbb{R}^J. \label{eq:eth2}
 \end{align}
Assume that the prior on $J$ and $\bm{\theta}$ satisfy Conditions (A1) and (A2). Let $J_n$, $\bar{J}_n \geq 2$ and $M_n$ be sequences of positive numbers such that the following hold for positive constants $a_3,a_4,c_3,c_4,C$ and any given constant $b>0$:
\begin{eqnarray}
&  b n\bar{\epsilon}_n^2 \le  J_n \log^{t_2}J_n,
~~ \log J_n +  n\bar{\epsilon}_n^2 \leq  M_n^{t_3}, \label{eq:e22} \\
& J_n \{(K_0+1) \log J_n + \log M_n +C_0 \log n\} \leq n \epsilon_n^2, \label{eq:a117} \\
&  \bar{J}_n^{-\alpha}  \leq \bar{\epsilon}_n, ~~\bar{J}_n \{c_1 \log^{t_1} \bar{J}_n + c_3 K_0 \log (\bar{J}_n) + c_3 \log(1/\bar{\epsilon}_n) \} \leq  2n \bar{\epsilon}_n^2, \label{eq:t333} \\
& \rho_n(w_1,w_2) \lesssim n^{a_3} \|w_1-w_2\|_r^{a_4}~\text{for any}~w_1,w_2 \in \mathcal{W}_{J_n,M_n},  \label{eq:rela}\\
& \max \left\{n^{-1} \sum_{i=1}^n K(p_{i,w_0},p_{i,w}), n^{-1} \sum_{i=1}^n  V(p_{i,w_0},p_{i,w}) \right\}
 \le C  \|w_1-w_2\|_r^2 , \label{eq:a114}
\end{eqnarray}
provided $\|w_1-w_2\|_r$ is sufficiently small. Then the posterior distribution of $w$ contracts at $w_0$ at the rate $\epsilon_n$ with respect to $\rho_n$.
\end{theorem}

\begin{proof}
In order to obtain the posterior contraction rate, we verify the following conditions as described in Theorem 4 of \citet{Ghosal20071}:
\begin{eqnarray}
& \log D(\epsilon_n,\mathcal{W}_{J_n,M_n} , \rho_n) \leq b_1 n\epsilon_n^2, \label{eq:a111} \\
& \Pi(w \notin \mathcal{W}_{J_n,M_n})  \leq b_3 \exp \{-n\epsilon_n^2\}, \label{eq:a112} \\
& \Pi(\mathcal{K}(w_0,\bar{\epsilon}_n)) \geq b_4 \exp \{-b_2 n \bar{\epsilon}_n^2\}, \label{eq:a1131}
\end{eqnarray}
where $\mathcal{W}_{J_n,M_n}$ is defined in Theorem \ref{thm:t0} and $b_1$, $b_2$, $b_3$, $b_4$ are some positive constants. Note that the conditions in Theorem \ref{thm:t0} are satisfied for $a_1=1$ and $a_2=2$ in the following way: \eqref{eq:e301n} is satisfied by the approximation assumption of $\bm{\xi}$; \eqref{eq:a117} implies \eqref{eq:e302n}; \eqref{eq:t333} implies \eqref{eq:e303n}; \eqref{eq:e304n} holds because of \eqref{eq:e22}. Using condition \eqref{eq:rela}, we obtain $$\log D(\epsilon_n,\mathcal{W}_{J_n,M_n},\rho_n)  \lesssim \log D(n^{-a_3} \epsilon_n^{a_4},\mathcal{W}_{J_n,M_n},\|\cdot\|_r) \lesssim n \epsilon_n^2$$ because $n^{-a_3} \epsilon_n^{a_4}$ is lower bounded by a polynomial in $n^{-1}$. Also, $ \Pi (w \notin \mathcal{W}_{J_n,M_n} )  \leq  2 \exp \{- n \epsilon_n^2\}$, therefore relation \eqref{eq:a112} holds for $b_3=2$. For \eqref{eq:a1131}, observe that
$
\Pi(\mathcal{K}(w_0,\bar{\epsilon}_n)) \geq \Pi(\|w - w_0\|_r \leq  \bar{\epsilon}_n)$
so the conclusion holds for an appropriate adjustment of constants in the definitions of the rates $\bar{\epsilon}_n$ and $\epsilon_n$.
\end{proof}

\begin{remark}\rm
For $r=2$ or $\infty$, relation \eqref{eq:eth1} holds for polynomials, Fourier series, B-splines and wavelets.
Relation \eqref{eq:eth2} holds for B-splines, polynomials and Fourier series base with $K_0=1/2$ when $r = 2$ and $K_0=1$ when $r = \infty$. For wavelets, \eqref{eq:eth2} holds with $K_0=1$ for $r=2,\infty$. This is because $\|(\bm{\theta_1} -\bm{\theta}_2)^T \bm{\xi}\|_p \leq \sum_{j=1}^J |\theta_{1j} - \theta_{2j}| \max_{1\leq j \leq J} \|\bm{\xi}_j\|_p \leq \sqrt{J} \|\bm{\theta}_1 -\bm{\theta}_2\|_2 C_{p,J}$ for $C_{p,J} = \max_{1\leq j \leq J} \|\bm{\xi}_j\|_p $ and $ 1 \leq p \leq \infty$. For B-splines, polynomials and Fourier series bases, $C_{p,J} \asymp 1$ when $p=2$ and $C_{p,J} \asymp \sqrt{J}$ when $p=\infty$. For wavelets, $C_{p,J} \asymp \sqrt{J}$ for $p=2,\infty$.
\end{remark}

\begin{remark}\rm\label{rm:nuis}
It is possible to incorporate a finite-dimensional nuisance parameter $\bm{\eta}$ in our setup, such as a scale parameter in a normal regression model. In this case, the sieve will be defined as the product of $\mathcal{W}_{J_n,M_n}$ with a suitable sieve for $\bm{\eta}$ whose metric entropy can be appropriately controlled and whose complement has exponentially small prior probability; see Remark \ref{sigmarm} for a concrete analysis.
\end{remark}

Theorem \ref{thm:t1} suggests that in order to obtain adaptive posterior contraction rates, it is crucial to choose sequences $J_n$, $\bar{J}_n$, $\epsilon_n$, $M_n$ in the rate equations \eqref{eq:e22}--\eqref{eq:t333} and bound the KL-divergences by the squared Euclidean distance $\|\cdot\|_r^2$.
Bounding the KL-divergence can be very different for various statistical problems, while the choices of $J_n$ and $\bar{J}_n$ are common for a set of basis functions. The following examples illustrate the use of the theorem.

\begin{example}[Fourier trigonometric series] \rm
For a function $w_0 \in \mathcal{C}^{\alpha}(0,1)$, the best approximation has the error $e(J) \asymp  J^{-\alpha} $ \citep{Dai2013}. Then the rate calculation proceeds in the following way: \eqref{eq:t333} implies $\bar{J}_n^{-\alpha} \lesssim \bar{\epsilon}_n$ and $\bar{J}_n \log n \lesssim n\bar{\epsilon}_n^2$, and hence $\bar{\epsilon}_n \asymp n^{-\alpha/(2\alpha+1)} (\log n)^{\alpha/(2\alpha+1)}$ and $\bar{J}_n \asymp (n/\log n)^{1/(2\alpha+1)} $. Now use \eqref{eq:e22}, we have $J_n \log^{t_2} n \gtrsim n\bar{\epsilon}_n^2$, hence we choose $J_n \asymp n^{1/(2\alpha+1)} (\log n)^{2\alpha/(2\alpha+1)-t_2}$. Note that \eqref{eq:a117} implies $J_n \log n \lesssim n \epsilon_n^2$. As a result, we choose $\epsilon_n \asymp n^{-\alpha/(2\alpha+1)} (\log n)^{\alpha/(2\alpha+1)+(1-t_2)/2}$.
\end{example}

\begin{example}[Bernstein polynomials] \rm
We consider the Bernstein polynomial prior proposed by \citet{Petrone1999b}. Consider a continuously differentiable density function $w_0$ with bounded second derivative, the approximation property of Bernstein polynomials to $w_0$ is $e(J) = C / J$ for some universal constant $C$ and $r=2$ \citep{Lorenz1953}. We can choose $\bar{J}_n = (n/\log n)^{1/3}$, $J_n=n^{1/3} (\log n)^{2/3 -t_2}$, $\bar{\epsilon}_n = (n/\log n)^{-1/3}$ and $M_n=n^{1/t_3}$. The rate $\epsilon_n$ is $n^{-1/3}(\log n)^{1/3 + (1-t_2)/2}$, which has the same polynomial power as given in \citet{Ghosal20012}. In fact, for any $0 \leq \alpha \leq 2$, the approximation rate of Bernstein polynomials is $J^{-\alpha/2}$ and the resulting posterior contraction rate is $n^{-\alpha/2(\alpha+1)}(\log n)^{\alpha/2(\alpha+1)}$; see \citet{Kruijer2008}. The poor contraction rate stems from the poor approximation rate of Bernstein polynomials. \citet{Kruijer2008} used coarsened Bernstein polynomials and showed that for any $f \in \mathcal{C}^{\alpha}(0,1)$ with $0 \leq \alpha \leq 1$, the approximation rate with $J$ undetermined parameters is $J^{-\alpha}$. If we choose $\bar{J}_n \asymp (n/\log n)^{-\alpha/(2\alpha+1)}$, then the rate is $\epsilon_n =n^{-\alpha/(2\alpha+1)} (\log n)^{\alpha/(2\alpha+1)+(1-t_2)/2}$, which adapts in the range $0 \leq \alpha \leq 1$.
\end{example}

\begin{example}[Polynomial basis] \rm
Consider the orthogonal Legendre polynomials as the approximation tool for $w_0 \in \mathcal{C}^{\alpha}(0,1)$. The rate of approximation is identical with that of the Fourier series under
the $L_2$- or the $L_{\infty}$-metrics (e.g., Theorem 6.1 of \citealp{Hesthaven2007}). Hence the choice of $J_n$, $M_n$ and rates are exactly the same with Example 1.
\end{example}

\begin{example}[B-splines] \rm
If we choose the B-spline functions (see Appendix) as the basis, then for $w_0 \in \mathcal{C}^{\alpha}(0,1)$, we have $e(J) \asymp J^{-\alpha}$ for either the $L_2$ or the $L_\infty$-distance. Thus the choices of the sequences and the resulting rate $\epsilon_n$ are the same as in the case polynomial or Fourier basis. However, one distinguishing property of the B-spline basis is the non-negativity of the basis functions so positive linear combinations are positive. Further we show in Appendix B that coefficients of a B-spline basis expansion can be restricted appropriately if the target function satisfies some restrictions. The property will allow some special prior distribution on the coefficients so that posterior moments can be calculated without using MCMC techniques.
\end{example}

\begin{example}[Wavelets] \rm
We consider a multiresolution truncated wavelet series
\begin{eqnarray}
\sum_{k=1}^{2^m-1} \alpha_{k} \phi_{k} (x) + \sum_{j=0}^{m} \sum_{k=1}^{2^m-1} \beta_{jk} \psi_{jk}(x),
\end{eqnarray}
where the boundary corrected wavelet basis of \citet{Cohen1993} is used since the domain is the unit interval, which results in a finite number of terms in the above expansion.
We put priors on $m$ and wavelet coefficients $\alpha_k$ and $\beta_{jk}$ for all possible values of $j,k$. It is well known that, for $w_0 \in \mathcal{C}^{\alpha}(0,1)$, the $L_2$-approximation error is $e(m) = 2^{-m \alpha}$. Hence we apply Theorem \ref{thm:t1} for $J = 2^{m}$ and choose $\bar{J}_n = (n/\log n)^{1/(2\alpha+1)}$, $J_n = n^{1/(2\alpha+1)} (\log n)^{2\alpha/(2\alpha+1) -t_2}$, $M_n=n^{1/t_3}$ and $\bar{\epsilon}_n = (n/\log n)^{-\alpha/(2\alpha+1)}$. Doing the same calculation as in Example 1, the resulting rate $\epsilon_n$ is $n^{-\alpha/(2\alpha+1)} (\log n)^{\alpha/(2\alpha+1)+(1-t_2)/2}$. This coincides with the adaptation results for white noise models in \citet{Lian2011} and for density estimation and regression models in \citet{Rivoirard2012}.
\end{example}

\begin{example}[Multivariate B-splines] \rm
Theorem \ref{thm:t0} can be used in multi-dimensional situation as well. Consider the tensor-product B-splines \citep{Schumaker2007} as a basis in $\mathcal{C}^{\alpha}(0,1)^s$. Then we have $e(J) \asymp J^{-\alpha/s}$ for $r=2$ or $\infty$, where $J=K^s$, and $K$ is the number of univariate B-spline functions used in making the tensor products. Apply Theorem \ref{thm:t1} with $\bar{J}_n = (n/\log n)^{1/(2\alpha+s)}$, $\bar{\epsilon}_n = (n/\log n)^{-\alpha/(2\alpha+s)}$, $M_n=n^{1/t_3}$ to obtain the rate $\epsilon_n$ as $n^{-\alpha/(2\alpha+s)}$ multiplied by some power of $\log n$, where the power depends on the statistical problem.
\end{example}

In these examples, we find that a power of $\log n$ is always present in the obtained rates. This is partly because we are dealing with a general class of problems. It is not clear whether such logarithmic terms can be removed and optimality can be established. Some negative results are given by \citet{Castillo2014}, where a sharp rate with a precise logarithmic term is obtained under $L_2$-loss. In some special situations, this logarithmic factor can be removed by using particular types of priors, such as \citet{Huang2004}, \citet{Ghosal2008} and \citet{Gao2013}.

\section{Density estimation}\label{sec5}
In this section, we illustrate how Theorem \ref{thm:t1} can be used to obtain adaptive posterior contraction rate for both the univariate and the multivariate density estimation where in the latter case the true density can be anisotropic, allowing different smoothness in different direction. We also discuss an MCMC-free method for calculating posterior moments by using a special conjugate-like prior on the coefficient vector.

\subsection{Univariate density estimation}
We consider estimation of a density defined on $(0,1)$. Frequentist optimal rate of contraction $n^{-\alpha/(2\alpha+1)}$ was obtained for the maximum likelihood estimators in \citet{Hasminskii1978}. A Bayesian method using a log-spline prior was studied in \citet{Ghosal2000}, where the optimal posterior contraction rate $n^{-\alpha/(2\alpha+1)}$ was obtained. When $\alpha$ is unknown, the adaptive posterior contraction rate $n^{-\alpha/(2\alpha+1)}$, possibly up to an additional logarithmic factor, was established in \citet{Ghosal2002,Ghosal2008}.

Consider estimating a density function $p$ on $(0,1)$. A prior can be induced on $p$ by using basis functions through a nonnegative, monotonic, locally Lipschitz continuous link function $\Psi$, i.e., $p_{w} = \Psi(w )/\int_0^1 \Psi\{w(x)\}dx$ for $ w = \bm{\theta}^T \bm{\xi}$, $\bm{\theta} \in \mathbb{R}^{J}$ and $J$ is given a prior on $\mathbb{N}$. If we choose $\Psi$ as the exponential function and $\bm{\xi}$ as the B-spline, then it gives the log-spline prior. We can also choose $\Psi$ as the identity function, and restrict the prior for $\bm{\theta}$ on $\Delta_J$ when using the B-spline basis, by Lemma \ref{lemma:l00}, part (d) in the Appendix.

\begin{corollary}
Suppose that we have i.i.d observations $X_1,\ldots,X_n$ generated from a density $p_0$, which satisfies $w_0=\Psi^{-1} (p_0) \in \mathcal{C}^{\alpha}(0,1)$ and that $w_0$ is bounded in $[\underline{M},\overline{M}]$ for some positive constants $\underline{M}$ and $\overline{M}$. We assume that the prior satisfies Conditions (A1) and (A2), and the basis $\bm{\xi}$ satisfies \eqref{eq:eth1} and \eqref{eq:eth2} with $r=\infty$. If either $\log \Psi$ is Lipscitz continuous or $c(w) = \int_0^1 \Psi\{w(x)\}dx > \underline{C} $ for some constant $\underline{C} > 0$,  then the posterior contraction rate is $\epsilon_n = n^{-\alpha/(2\alpha+1)} (\log n)^{\alpha /(2\alpha+1)+(1-t_1)/2} $ at $p_0$ with respect to the Hellinger distance.
\end{corollary}

\begin{proof}
If $w$ is uniformly close to $w_0$, $\|\Psi(w)-\Psi(w_0)\|_\iy$ is small and hence $c(w)=\int \Psi(w(x))dx$ is close to $\int \Psi(w_0(x))dx=\int p_0(x)dx=1$, and hence is bounded below. Thus we have the estimate
\begin{eqnarray}
\|p_w - p_0\|_{\infty}  & \leq & \left\| \frac{\Psi(w)}{c(w)} -\Psi(w)\right\|_{\infty} + \|\Psi(w) - \Psi(w_0) \|_{\infty} \nonumber \\
 & \leq & (c(w))^{-1} |c(w) - c(w_0) | \|\Psi(w)\|_{\infty} + \|\Psi(w) - \Psi(w_0) \|_{\infty} \nonumber \\
 & \lesssim & \|\Psi(w) - \Psi(w_0)\|_{\infty} \lesssim \|w-w_0\|_{\infty}.
\end{eqnarray}
 Note that because $p_0$ is bounded away from $0$, so is $p_w$ when $\|w-w_0\|_\iy$ is small. Now
\begin{eqnarray}
h^2(p_0,p_w) = \int \frac{|p_0-p_w|^2}{(\sqrt{p_0}+\sqrt{p_w})^2} \leq \frac{1}{\underline{M}} \|p_0 - p_w\|_{\infty}^2 \lesssim \|w-w_0\|_{\infty}^2.
\end{eqnarray}
Using Lemma 8 of \citet{Ghosal2007}, we have
\begin{eqnarray}\label{eq:e401}
K(p_0,p_w) \leq 2h^2(p_0,p_w)\Big \| \frac{p_0}{p_w} \Big \|_{\infty} \lesssim \|w-w_0\|_{\infty}^2 , \nonumber \\
V(p_0,p_w) \lesssim h^2(p_0,p_w) \Big( 1 + \Big \| \frac{p_0}{p_w} \Big \|_{\infty} \Big)^2 \lesssim \|w-w_0\|_{\infty}^2 .
\end{eqnarray}

Therefore \eqref{eq:a114} holds for $r=\infty$. Next, we verify \eqref{eq:rela}.  Note that because of the i.i.d assumption, $\rho_n$ is the Hellinger distance on $p_w$.
As the Hellinger distance is bounded by the square root of the $L_1$-distance, it suffices to bound the latter. If $\log \Psi$ is Lipschitz continuous with Lipschitz constant $L$, bound  $\|p_{w_1}-p_{w_2}\|_1$ by
$$
2 \frac{\|\Psi(w_1)-\Psi(w_2)\|_1}{c(w_1)}
\le  2\|\exp[\log \Psi(w_1)-\log\Psi(w_2)]-1\|_\iy \le 2L \|w_1-w_2\|_\iy e^{L\|w_1-w_2\|_\iy}.$$
On the other hand if $c(w)$ is bounded below by $\underline{C}$, we obtain
$$\|p_{w_1}-p_{w_2}\|_1\le 2 \frac{\|\Psi(w_1)-\Psi(w_2)\|_1}{c(w_1)} \lesssim \|\Psi(w_1) - \Psi(w_2)\|_1 \lesssim \|w_1 - w_2\|_{\infty}, $$
so that the assertion holds with $a_4=1/2$.

Now we apply Theorem \ref{thm:t1} with $\bar{J}_n = (n/\log n)^{1/(2\alpha+1)}$,
 $J_n = n^{1/(2\alpha+1)} (\log n)^{2\alpha/(2\alpha+1) -t_2}$, $\bar{\epsilon}_n = (n/\log n)^{-\alpha/(2\alpha+1)}$, $M_n=n^{1/t_3}$ and $r=\infty$,
then the posterior distribution contracts at the rate $\epsilon_n = n^{-\alpha/(2\alpha+1)} (\log n)^{\alpha /(2\alpha+1)+(1-t_2)/2} $ at $p_0$ with respect to the Hellinger distance.
\end{proof}

\begin{remark}\rm
The commonly used exponential link function trivially satisfies the first requirement that $\log \Psi$ is Lipschitz continuous. The identity link function, used for MCMC-free calculation in Section~3.3 satisfies $c(w)=\int w(x)dx=1$ as $w$ is a probability density, and hence the condition that $c(w)$ bounded away from zero trivially holds.
\end{remark}

\subsection{Anisotropic multivariate density estimation}\label{mulden}
We extend univariate density estimation to the multivariate situation by considering estimating an $s$-dimensional density function on $(0,1)^s$. We induce a prior on the density using through the relation $p_{w} \propto \Psi(\bm{\theta}^T \bm{\xi} )$ with $\bm{\xi}$ chosen as the tensor-product B-spline basis of order $q$. The true density $p_0$ is allowed to have different smoothness levels at different directions. More precisely, we define an anisotropic H\"older smoothness class by
     \begin{eqnarray*}
    \mathcal{C}^{\bm{\alpha}}(0,1)^{s}= \left\{f(x_1,\ldots,x_s): \left\|\frac{\partial^{\sum_{k=1}^s l_k}f}{\partial x_1^{l_1} \cdots \partial x_s^{l_s}} \right\|_{\infty} < \infty,~ 0 \leq l_k \leq \alpha_k,~ k=1,\ldots,s,~\sum_{k=1}^s l_k/\alpha_k<1.\right\}
    \end{eqnarray*}
 for some smoothness parameter $\bm{\alpha} = (\alpha_1,\ldots,\alpha_s)$, which are integers not greater than $q$. Let $J(1),\ldots,J(s)$ be the number of basis functions for individual $s$ directions and define $J$ as their products. Given $\Psi^{-1} (p_0) \in \mathcal{C}^{\bm{\alpha}}(0,1)^{s}$, the approximation error is of the order $\sum_{k=1}^s J(k)^{-\alpha_k}$ according to Theorem 12.7 of \citet{Schumaker2007}. Hence for the best balancing of the approximation error,we choose $J_{n}(k) = \bar\epsilon_n^{-1/\alpha_k}$ and $\bar{J}_n = \prod_{k=1}^s J_{n}(k) = \bar\epsilon_n^{-s/\alpha^*}$, where $\alpha^* = s/(\sum_{k=1}^s \alpha_k^{-1})$ is the harmonic mean of $\alpha_1,\ldots,\alpha_s$, and $\bar{\epsilon}_n$  is to be chosen to match $\bar{\epsilon}_n^{-s/\alpha^*}\log n$ with $n\bar{\epsilon}_n^2$. Applying Theorem \ref{thm:t1} with $\bar{\epsilon}_n = (n/\log n)^{-\alpha^*/(2\alpha^*+s)}$, $J_n = n^{s/(2\alpha^*+s)} (\log n)^{(2\alpha^*)/(2\alpha^*+s) -t_2}$, $M_n=n^{1/t_3}$, $a_4=1/2$ and $r=\infty$, the posterior distribution contracts at $p_0$ with respect to the Hellinger distance at the rate $\epsilon_n = n^{-\alpha^*/(2\alpha^*+s)} (\log n)^{\alpha^* /(2\alpha^* +s)+(1-t_2)/2} $. Essentially the same rate is also obtained in \citet{Shen2013} (with a different logarithmic factor) using a Dirichlet mixture of normal prior.

\subsection{MCMC-free computation}
Next, we describe an MCMC-free calculation technique for the univariate density estimation using normalized B-splines $\{B_1^*,\ldots,B_J^*\}$ as the basis; see Appendix.
By part (d) in Lemma \ref{lemma:l00}, we can restrict the coefficients $\bm{\theta}$ to a $J$-dimensional simplex $\Delta_J$ and maintain the same approximation rate. We put a Dirichlet prior on $\bm{\theta} \sim \text{Dir}(a_1,a_2,\ldots,a_J)$ for any $J \in \mathbb{N}$. Finally, we assign a prior $\Pi$ on $J$. Thus a prior on the density $p$ is induced. Given the observations $\bm{X}=(X_1,\ldots,X_n)$ and a fixed dimension $J$, the posterior density of $\bm{\theta}$ is a mixture of Dirichlet distribution:
\begin{eqnarray}
p(\bm{\theta} |\bm{X},J) & \propto & \prod_{k=1}^J \theta_k^{a_k-1} \prod_{i=1}^n \big\{\sum_{k=1}^J \theta_k B_k^*(X_i)\big\} = \sum_{i_1=1}^J \cdots \sum_{i_n=1}^J \prod_{k=1}^J \theta_k^{a_k-1} \prod_{s=1}^n \theta_{i_s} B_{i_s}^*(X_s). \nonumber
\end{eqnarray}
Using $p(J,\bm{\theta}|\bm{X}) \propto p(\bm{X}|J,\bm{\theta}) \Pi(\bm{\theta}|J) \Pi(J)$,
 the posterior mean of $p$ at a point $x$ is
\begin{eqnarray}\label{eq:e1323}
\lefteqn{\displaystyle{\frac{\sum_{j=1}^{\infty} \int_{\bm{\theta}} p(x) p(\bm{X}|J=j,\bm{\theta}) \Pi(\bm{\theta}|J=j) \Pi(J=j) d\bm{\theta} }
{\sum_{j=1}^{\infty} \int_{\bm{\theta}} p(\bm{X}|J=j,\bm{\theta}) \Pi(\bm{\theta}|J=j) \Pi(J=j) d\bm{\theta}} }} \nonumber \\
&=\displaystyle{\frac{\sum_{j=1}^{\infty} \Pi(j) \sum_{i_0=1}^j \sum_{i_1=1}^j \cdots \sum_{i_n=1}^j \int_{\bm{\theta} \in \Delta_j} \prod_{k=1}^j \theta_k^{a_k-1} \prod_{s=0}^n \theta_{i_s} B_{i_s}^*(X_s) d\bm{\theta} }{ \sum_{j=1}^{\infty} \Pi(j) \sum_{i_1=1}^j \cdots \sum_{i_n=1}^j \int_{\bm{\theta} \in \Delta_j} \prod_{k=1}^j \theta_k^{a_k-1} \prod_{s=1}^n \theta_{i_s} B_{i_s}^*(X_s) d\bm{\theta} }}, &
\end{eqnarray}
where $X_0$ stands for $x$. Define $I_{k,j,0}^{\bm{i}}=\sum_{s=0}^n \Ind \{i_s=k\}$ and $I_{k,j,1}^{\bm{i}}= \sum_{s=1}^n \Ind\{i_s=k\}$. Then the expression in \eqref{eq:e1323} can be simplified to
\begin{eqnarray}\label{eq:e901}
 \frac{ \displaystyle{\sum_{j=1}^{\infty} \Pi(j) \sum_{i_0=1}^j \sum_{i_1=1}^j \cdots \sum_{i_n=1}^j \prod_{k=1}^j \Gamma(a_k + I_{k,j,0}^{\bm{i}}) \prod_{s=0}^n B_{i_s}^*(X_s) / \Gamma \Big(\sum_{i=1}^j a_i + n +1 \Big)}} { \displaystyle{\sum_{j=1}^{\infty} \Pi(j)  \sum_{i_1=1}^j \cdots \sum_{i_n=1}^j \prod_{k=1}^j \Gamma(a_k + I_{k,j,1}^{\bm{i}}) \prod_{s=1}^n  B_{i_s}^*(X_s) / \Gamma \Big(\sum_{i=1}^j a_i + n \Big) }}.
\end{eqnarray}
A basis function takes nonzero values only at $q$ intervals, so the calculation involves a multiple of $q^{n+1}$ steps. More details are given in Section \ref{sec9}. Similar expressions can be obtained for other posterior moments, in particular, for the posterior variance.

Note that if $q=1$, the sums over indices $i_1,\ldots,i_n$ in \eqref{eq:e901} will be redundant, leading to a histogram estimate whose bin length and weights are posterior averaged. The B-spline random series prior can also be viewed as a kernel mixture prior, where the kernel is a B-spline function indexed by a discrete parameter.

For multivariate situation, MCMC-free computational techniques can be developed in a similar way using tensor products of normalized B-splines as the basis and a Dirichlet prior on the corresponding coefficients. The approximation property established in the last part of Lemma~\ref{lemma:l00*} justifies restricting the coefficients on the simplex.

\section{Regression models}\label{sec:reg}
In this section, we consider several nonparametric regression problems including regression with additive Gaussian errors, binary regression, Poisson regression and functional regression. In these cases, we allow the covariates be either fixed or random and show how Theorem \ref{thm:t1} can be used to derive contraction rates. The techniques also apply for multivariate analogs of these regression problems using the tensor-product B-spline basis as in Subsection \ref{mulden}. 

For fixed covariates $\bm{Z}$, define the empirical measure $\mathbb{P}_n^Z = n^{-1} \sum_{i=1}^n \delta_{Z_i}$, and $\|\cdot \|_{2,n}$ as the norm on $L_2(\mathbb{P}_n^Z)$.

\subsection{Nonparametric regression with Gaussian errors}\label{examplenormreg}
We consider a regression model with additive error $X_i = f(Z_i)+\varepsilon_i$, where $\varepsilon_i \stackrel{\text{iid}}{\sim} \text{N}(0,\sigma^2)$, $Z_1,\ldots,Z_n \in (0,1)$.  For ease of illustration, we first consider known $\sigma$ and fixed covariates; the modification necessary for unknown $\sigma$ and random covariates is outlined in Remark~\ref{sigmarm}. 

\begin{corollary}
Suppose that the true regression function $f_0 \in \mathcal{C}^{\alpha}(0,1)$ and the prior satisfies Conditions (A1) and (A2).
Given fixed covariates, assume that the basis $\bm{\xi}$ satisfies \eqref{eq:eth1} and \eqref{eq:eth2} with $r=\infty$. Then the posterior of $f$ contracts at the rate
$\epsilon_n = n^{-\alpha/(2\alpha+1)} (\log n)^{\alpha/(2\alpha+1)+(1-t_2)/2} $ relative to $\|\cdot\|_{2,n}$ at $f_0$.
\end{corollary}

\begin{proof}
Let $P_{f,i}$ be the normal measure with mean $f(Z_i)$ and variance $\sigma^2$. Then the Hellinger distance between $P_{f_1,i}$ and $P_{f_2,i}$
is of the order of the $|f_1(Z_i)-f_2(Z_i)|$ when one is, and hence both are,
small. Hence the conclusions of Lemma 2 of \citet{Ghosal20071} hold (with different
constants) for the distance $\|\cdot\|_{2,n}$. This implies that to compute entropy we can work with $\|\cdot\|_{2,n}$ instead of $\rho_n$. Using the arguments in Section 7.2 of \citet{Ghosal20071}, we get
\begin{equation}\label{eq:e703}
\max \Big\{ n^{-1} \sum_{i=1}^n K(P_{f_0,i}, P_{f,i}), n^{-1} \sum_{i=1}^n V(P_{f_0,i}, P_{f,i}) \Big\} \leq \|f_0 - f\|_{2,n}^2 /\sigma^2 \leq \|f_0 - f\|_{\infty}^2/\sigma^2.
\end{equation}

Clearly, Condition \eqref{eq:rela} holds for $\|\cdot\|_{2,n}$ with $r=\infty$. Assuming that the basis $\bm{\xi}$ satisfies \eqref{eq:eth1} and \eqref{eq:eth2} with $r=\infty$ and choosing $\bar{J}_n = (n/\log n)^{1/(2\alpha+1)}$, $\bar{\epsilon}_n = (n/\log n)^{-\alpha/(2\alpha+1)}$, $J_n = n^{1/(2\alpha+1)} (\log n)^{2\alpha/(2\alpha+1) -t_2}$, $r=\infty$ and $M_n=n^{1/t_3}$, then we obtain the posterior contraction rate $\epsilon_n = n^{-\alpha/(2\alpha+1)} (\log n)^{\alpha/(2\alpha+1)+(1-t_2)/2} $ relative to $\|\cdot\|_{2,n}$.
\end{proof}

\begin{remark}\rm\label{sigmarm}
For random covariates $Z_1,\ldots,Z_n \sim G$, define $L_{2,G}$ as the $L_2$-distance with respect to the probability measure $G$. We assume that $G$ has a density $g$ bounded and bounded away from zero, and $\bm{\xi}$ satisfies \eqref{eq:eth1} and \eqref{eq:eth2} with $r=2$. Then $\|f_1 - f_2 \|_{2,G}$ is equivalent to $\|f_1 - f_2\|_{2}$, and hence can be used interchangeably in entropy calculations and bounding prior concentration and posterior contraction rates. Alternatively without any conditions on $G$, we can assume the basis has the $L_{\infty}$-approximation property with the same rate and bound $\|\cdot\|_{2,G}$ by $\|\cdot\|_{\infty}$. Hence by applying Theorem \ref{thm:t1} in the same way with $r=\infty$, we obtain the same rate with respect to $\| \cdot \|_{2,G}$.

When $\sigma$ is unknown, we assign a prior (independent of other parameters) on it. If the prior density is positive throughout, and has exponential tail near zero and polynomial tail near infinity, then a sieve $(n^{-C_1},\exp\{C_2 n \epsilon_n^2\} )$ with sufficiently large $C_1,C_2$ will satisfy the conditions in Theorem \ref{thm:t1}. Note that the popular inverse gamma prior on $\sigma^2$ (or on any positive power of $\sigma$) satisfies the requirements.
\end{remark}

\subsection{Nonparametric binary regression}
Assume that we have $n$ independent observations $(Z_1,X_1),\ldots,(Z_n,X_n)$ from a binary regression model $\mathrm{P}(X=1|Z=z)=1-\text{P}(X=0|Z=z)=f_0(z)$, where $X$ takes values in $\{0,1\}$ and $Z$ is either a fixed or a random covariate in some domain $\mathcal{Z}$. Given a link function $\Psi: \mathcal{Z} \rightarrow (0,1)$, we can construct a random series prior on the regression function $f_0$ using a basis $\bm{\xi}$ as $f_{\bm{\theta}}(z) = \Psi \{\bm{\theta}^T \bm{\xi}(z)\}$. Commonly, a cumulative distribution function on $\mathbb{R}$ such as the logit or probit function is chosen as the link function and the coefficient vector $\bm{\theta}$ can take any values in $\mathbb{R}^J$. Then any basis with approximation property for the H\"{o}lder class may be used.


\begin{corollary}
Suppose that the true classification function $f_0$ is bounded away from $0$ and $1$, and satisfies $w_0=\Psi^{-1} (f_0) \in \mathcal{C}^{\alpha}(0,1)$. Given fixed covariates, and that the prior satisfies Conditions (A1) and (A2). Assume that the basis $\bm{\xi}$ satisfy \eqref{eq:eth1} and \eqref{eq:eth2} with $r=\infty$, and the link function $\Psi$ is Lipschitz continuous. Then the posterior of $f$ contracts at the rate $\epsilon_n = n^{-\alpha/(2\alpha+1)} (\log n)^{\alpha/(2\alpha+1)+(1-t_2)/2} $ relative to $\|\cdot\|_{2,n}$ at $f_0$.
\end{corollary}



\begin{proof}
Define $p_{w} = \Psi(w)^x (1-\Psi(w))^{1-x}$, note that by the Lipschitz continuity of $\Psi$,
\begin{eqnarray*}
h^2(p_{w_1}, p_{w_2})\le \|p_{w_1} - p_{w_2} \|_1 = 2 \|\Psi(w_1) - \Psi(w_2)\|_{\infty}\lesssim \|w_1-w_2\|_\iy, \\
\max\{K(p_{w_0},p_w),V(p_{w_0},p_w) \}  \lesssim \|\Psi(w) - \Psi(w_0)\|_{\infty}^2 \lesssim \|w - w_0\|_{\infty}^2,
\end{eqnarray*}
so the relation \eqref{eq:rela} holds with $a_4=1/2$. Now we may apply Theorem \ref{thm:t1} with $\bar{J}_n = (n/ \log n)^{1/(2\alpha+1)}$, $M_n=n^{1/t_3}$, $J_n = n^{1/(2\alpha+1)} (\log n)^{2\alpha/(2\alpha+1)  -t_1}$ and $\bar{\epsilon}_n = n^{-\alpha/(2\alpha+1)} (\log n)^{(\alpha+1)/(2\alpha+1)}$, then the posterior distribution contracts at the rate $\epsilon_n = n^{-\alpha/(2\alpha+1)} (\log n)^{\alpha /(2\alpha+1)+(1-t_2)/2} $ relative to $\rho_n$. By Taylor's expansion of the squared Hellinger distance in a binomial model, it is easy to see that $\rho_n$ is equivalent with the $\|\cdot\|_{2,n}$-distance on $f$.
\end{proof}

For random covariates $Z_1,\ldots,Z_n \sim G$, when $G$ has a density $g$ bounded and bounded away from zero, the same conclusion can be made in terms of the $L_2$-distance on $f$, or more generally, with respect to the $L_2(G)$-distance without any additional conditions.

When specifically the B-splines basis is used, the link function $\Psi$ can be chosen to be the identity function in view of part (c) of Lemma \ref{lemma:l00}. The expressions then simplify significantly if we use beta priors $\theta_i \stackrel{\text{ind}}{\sim} \text{Beta}(a_i,b_i)$ for some positive numbers $a_i$ and $b_i$. 

\subsection{Nonparametric Poisson regression}
Consider a Poisson regression model $X_i \stackrel{\text{ind}}{\sim} \text{Poi}\{f(Z_i)\}$, where $f$ is an unknown monotonic function and $Z$ is a covariate. For convenience, we assume that $Z$ takes values in $(0,1)$. Using a random series expansion, $f$ can be modeled through a link function $f(z)=\Psi(\bm{\theta}^T\bm{\xi})(z)$.

\begin{corollary}
Suppose that $\Psi^{-1}(f_0) \in \mathcal{C}^{\alpha}(0,1)$ and $f_0$ is bounded away from zero and infinity. Let the prior satisfy Conditions (A1) and (A2). Assume that the basis $\bm{\xi}$ satisfies \eqref{eq:eth1} and \eqref{eq:eth2} with $r=\infty$, and the link function $\Psi$ is monotonic and Lipschitz continuous on $(0,1)$ and $\sqrt{\Psi}$ is Lipschitz continuous on $[1,\iy)$.  Then the posterior of $f$ contracts at the rate $\epsilon_n = n^{-\alpha/(2\alpha+1)} (\log n)^{\alpha/(2\alpha+1)+(1-t_2)/2} $ relative to root-average squared Hellinger distance $\rho_n$ at $f_0$.
\end{corollary}

\begin{proof}
In a Poisson model, the squared Hellinger distance is easily bounded by twice the square of the difference of the square roots of the parameters when the parameters are in $[1,\iy)$, while it is bounded by the $L_1$-distance which is further bounded by the absolute difference of the parameters when they lie in $(0,1)$. Hence by the Lipschitz continuity assumptions on $\Psi$, with the choice $r=\iy$, the condition  \eqref{eq:rela} holds with $a_3=0$ and $a_4=1/2$ when $\|w_1-w_2\|_\iy$ is small. In fact, it is sufficient to assume that the Lipschitz continuity condition on $\sqrt{\Psi}$ holds with the Lipschitz constant growing up to polynomially in $n$ on a sieve $\{\bm{\th}^T \bm{\xi} \ge 1$, $\|\bm{\th} \|_\iy\le n^c$ and $J_n\le n\}$. The Kullback-Leibler divergences in  Poisson model near a positive value of the parameter are bounded by a multiple of the square of the difference of parameter values, and a fixed constant can be chosen uniformly for all true parameter values lying in a compact subset of $(0,\iy)$. This leads to the verification of \eqref{eq:a114}.

For any of the discussed basis functions, an application of Theorem \ref{thm:t0} with the $L_{\infty}$-distance verifies the remaining conditions of  Theorem \ref{thm:t1} for $\bar{J}_n = (n/ \log n)^{1/(2\alpha+1)}$, $M_n=n^{1/t_3}$, $J_n = n^{1/(2\alpha+1)} (\log n)^{(2\alpha+2)/(2\alpha+1)  -t_1}$ and $\bar{\epsilon}_n = n^{-\alpha/(2\alpha+1)} (\log n)^{(\alpha+1)/(2\alpha+1)}$, then the posterior contraction rate is obtained as $\epsilon_n = n^{-\alpha/(2\alpha+1)} (\log n)^{\alpha/(2\alpha+1) +(1-t_1)/2}$ relative to $\rho_n$. 
\end{proof}

To reinterpret this contraction rate in terms of the more desirable $\|\cdot\|_{2,n}$-distance on $f$, we observe that these two are equivalent near the true regression function $f_0$ by its positivity and boundedness properties, provided that $f$ remains in an $L_\iy$-bounded set with high posterior probability for most samples drawn from the true distribution. This is obviously ensured if coefficients get a prior confined in a bounded set, but will also hold if the posterior is consistent for the $L_\iy$-distance on $f$.

For random covariates $Z_1,\ldots,Z_n \sim G$, the same contraction rate is obtained with respect to the Hellinger distance on the joint density of $(X,Z)$, and with respect to the $L_2(G)$-distance on $f$ under the aforementioned additional conditions.

If we use B-splines to form the basis, in view of Part (c) of Lemma \ref{lemma:l00}, we are allowed to restrict $\theta_j$ to positive values. By choosing the identity link, then it is possible to carry out MCMC-free computation by letting $\theta_i \stackrel{\text{ind}}{\sim} \text{Gamma}(a_i,b_i)$ for some positive numbers $a_i$ and $b_i$. The resulting prior satisfies all requirements for the posterior contraction rate obtained above.

\subsection{Functional regression model}
Spline functions are widely used to model functional data; see \citet{Cardot2003} for example. A rate of contraction result was obtained in \citet{Hall2007}. A Bayesian method based on splines was given by \citet{Goldsmith2011}. However, to the best of our knowledge, no results on posterior contraction rates for Bayesian methods are yet available. We consider two types of functional regression model. The first one assumes only the covariates $Z(t)$ and the effects $\beta(t)$ depend on time $t$. The second one allows functional observations $X(t)$. We can use any basis with general approximation properties for H\"{o}lder classes under the $L_2$-distance.

We first discuss the case of functional covariates with a scalar response. Suppose we observe i.i.d. copies  $(Z_1,X_1),\ldots,(Z_n,X_n)$ of $(Z,X)$, where $Z$ is a square integrable random function  defined on $(0,1)$ and $X$ is a scalar. A functional linear regression model can be formulated as follows:
\begin{eqnarray}\label{funcscalar}
X_i = \int_0^1 Z_i(t) \beta(t) dt + \varepsilon_i,
\end{eqnarray}
where $\beta(t)$ is the coefficient function we want to estimate, $\varepsilon_1,\ldots,\varepsilon_n \stackrel{\text{iid}}{\sim} \text{N}(0,\sigma^2)$. We consider $\sigma$ to be known; the more realistic case of unknown $\sigma$ can be treated following Remark~\ref{sigmarm}.

\begin{corollary}
Suppose that the true regression function $\beta \in \mathcal{C}^{\alpha}(0,1)$, $\mathrm{E} Z^2(t)$ is uniformly bounded away from $0$ and $\infty$ for every $t \in (0,1)$, and the basis satisfies \eqref{eq:eth1} and \eqref{eq:eth2} with $r=2$. Given the prior being constructed as in (A1) and (A2), the posterior of $f$ contracts in a rate $\epsilon_n = n^{-\alpha/(2\alpha+1)} (\log n)^{\alpha/(2\alpha+1)+(1-t_2)/2} $ relative to the $L_2$-distance.
\end{corollary}

\begin{proof}
We consider a basis expansion $\beta(t)=\sum_{k=1}^J \theta_k \xi_k(t)$. Denote $W_{ik}=\int_{0}^1 Z_i(t) \xi_k(t) dt$, then the model can be written as
$
X_i = \sum_{k=1}^J \theta_k W_{ik} + \varepsilon_i.
$
Define $P_{\beta}(\cdot|Z)$ as the normal measure with mean $\int_0^1 Z(t) \beta(t) dt$ and variance $\sigma^2$, and let $\text{E}_Z$ be the expectation with respect to the distribution of $Z$. Then we can bound $K(P_{\beta_0},P_{\beta})$ and $V(P_{\beta_0},P_{\beta})$ using Cauchy-Schwarz inequality:
\begin{eqnarray*}
 \max \big\{   K(P_{\beta_0}, P_{\beta}),  V(P_{\beta_0}, P_{\beta}) \big\}  \lesssim   \frac{1}{\sigma^2}\text{E}_{Z} \Big( \int_0^1 Z(t) \{\beta(t) - \beta_0(t)\} dt \Big)^2 \lesssim \frac{1}{\sigma^2} \|\beta-\beta_0\|_2^2.
\end{eqnarray*}
For \eqref{eq:rela}, note that the same argument used in random covariates situation in Section \ref{examplenormreg} applies here. Hence we can apply Theorem \ref{thm:t1} as in Section \ref{examplenormreg}. Then the posterior contracts at the rate $\epsilon_n = n^{-\alpha/(2\alpha+1)} (\log n)^{\alpha /(2\alpha+1)+(1-t_2)/2} $ relative to the $L_2$-distance.
\end{proof}

Next, we consider a longitudinal type of functional model:
\begin{eqnarray}
X_i(T_i)=Z_i(T_i) \beta(T_i) + \varepsilon_i.
\end{eqnarray}
For each object $i$, we observe its response $X_i$ at a random time $T_i \in (0,1)$ with a random covariate $Z_i$. We assume that $Z_1,\ldots,Z_n$ are i.i.d. copies of $Z$, $T_1,\ldots,T_n$ are i.i.d. copies of $T$, $\varepsilon_i \stackrel{\text{iid}}{\sim} \text{N}(0,\sigma^2)$, they are all independent of each other and $T$ has a density $g$ bounded and bounded away from zero on $(0,1)$. Again it suffices to treat $\sigma$ as known.

Suppose that the true regression function $\beta \in \mathcal{C}^{\alpha}(0,1)$ and $\mathrm{E} Z^2(t)$ are uniformly bounded away from $0$ and $\infty$ for every $t \in (0,1)$. Then again
\begin{eqnarray*}
 \max \big\{   K(P_{\beta_0}, P_{\beta}),  V(P_{\beta_0}, P_{\beta}) \big\} \lesssim   \frac{1}{\sigma^2} \text{E}  \int_0^1 Z^2(t) (\beta(t) - \beta_0(t))^2 g(t) dt \lesssim \frac{1}{\sigma^2}  \|\beta - \beta_0\|_2^2.
\end{eqnarray*}
Hence we obtain the same contraction rate $\epsilon_n$ if we use the same prior on $\beta$ as before. This rate coincides with the optimal rate obtained in \citet{Cai2011} within a logarithmic factor.

\section{Numerical examples}\label{sec9}
\subsection{Simulation}
We illustrate the use of conjugate prior structure as described in \eqref{eq:e1323} and \eqref{eq:e901} on density estimation problems. We consider two examples of the true density: Beta$(0.5,0.5)$, and a mixture density of exponential and a normal distribution:
\begin{eqnarray}\label{mixden}
f_0(x) \propto \frac{3}{4}3e^{-3x} + \frac{1}{4}\frac{\sqrt{32}}{\sqrt{\pi}} e^{-32(x-0.75)^2}.
\end{eqnarray}
For each density, we generate $n=20$, $50$, $100$ and $300$ samples and then implement the random series prior for $q=1$ and $q=3$. When $q=1$, the exact value of the posterior mean can be calculated. When $q=3$, instead of evaluating all possible terms to get \eqref{eq:e901}, we randomly sample $N=3000$ of them and take the associated average values. We choose a geometric prior for $J$ restricted between $5$ and $25$. The lower truncation ensures a minimum number of terms in the series expansion while an upper truncation is necessary to carry out the actual computation using a computer. For $\bm{\theta}$, we use the uniform distribution on the simplex as a default choice for the Dirichlet distribution. We obtain density estimates at $100$ grid points in the unit interval.

We compare our results with that using the Gaussian process (GP) prior in \citet{Tokdar2007} and Dirichlet mixture (DM) of normal kernels \citep{Escobar1995}. Mean absolute errors, mean squared errors (note that the theoretical results are obtained for Hellinger distance though) and computing time (in seconds) are summarized in Table \ref{table:mixture}. Standard errors (s.e.) are calculated based on $100$ Monte-Carlo replications. Comparing the performance of RSP using $q=1$ with that of $q=3$, we observe a trade-off between computation time and estimation accuracy. In terms of estimation accuracy, RSP ($q=3$) beats DM in both cases, but performs  worse than GP for the mixture density estimation. Overall, RSP ($q=1$) has the lowest computation cost due to its simple expression. It will be interesting to consider a utility function that simultaneously evaluates the performance of estimators based on time and accuracy \citep{Asm2007}.

 Note that for RSP, the computational complexity becomes exponential in $n$ given $q > 1$, and hence all terms in the posterior mean cannot be computed for larger values of $n$. In this situation, we sample and compute a manageable number of terms and estimate the total as in sample survey for finite populations. The resulting standard error for sampling can be estimated in the usual way from the computed terms, and will be often reasonable if the terms are not very unlike each other.


We also calculate pointwise credible bands ($95\%$ nominal coverage) for the mixture true density example \eqref{mixden} based on the second moment estimation. Results are given in Figure \ref{fig:den} for smoothness level $q=1,3$ and sample size $n=100$ and $500$. There is a significant improvement by using higher values of $q$. Bernstein-von Mises results provided by \citet{Rivoirard20121} may be useful in establishing frequentist coverage properties of these intervals.

\begin{table}[H]
  \caption{Density estimation results: mean squared error ($l_2$), mean absolute error ($l_1$), and computational time in seconds ($t$), using random series priors (RSP) with $q=1$ and $3$, Gaussian process (GP) and Dirichlet mixture (DM) priors. }  \label{table:mixture}
  \begin{center}
  \begin{tabular*}{1\textwidth}{@{\extracolsep{\fill}}lc|cccccccccccc}
    \hline
    & & \multicolumn{3}{c}{$n = 20$}&\multicolumn{3}{c}{$n=50$}&\multicolumn{3}{c}{$n=100$}&\multicolumn{3}{c}{$n=300$}\\  \cline{3-5} \cline{6-8} \cline{9-11} \cline{12-14}
     True density &  &$l_2$&$l_1$&$t$&$l_2$&$l_1$&$t$&$l_2$&$l_1$&$t$&$l_2$&$l_1$&$t$\\ \hline
   Mixture &  RSP ($q=1$) & .27 & .40 & .44 & .20 & .33 & .58 & .18 & .30 & .69 &.17 & .29 &1.24\\
   &  RSP ($q=3$) & .16 & .31 & 255 &.11 & .25 & 317 &.10 & .24 & 320 &.09 & .22 & 425\\
   &  GP  & .11 & .23 & 53.9  &.06 & .17 & 58.8 &.04 &.14 & 61.5 &.02 &.10 &66.9\\
   &  DM & .46 & .59 &9.6 &.28 &.44 &22.0 &.17 &.34 &33.6 &.11 &.28 &99.5\\
   &  max s.e. & .01 & .01 & - & .01 & .01 & - & .01 & .01 & - & .00 & .00 & - \\ \hline
 Beta$(0.5, 0.5)$ &  RSP ($q=1$) & .35 & .45 & .45 & .31 & .42 & .57 & .27 & .39 & .67 & .25 & .37 & 1.23\\
   &  RSP ($q=3$) & .16 & .27 & 267 & .15 & .27 & 314 & .14 & .25 & 324 & .11 & .22 & 428\\
   &  GP  & .34 & .39 & 55.7 &  .27 & .34 & 61.6 & .24 & .31 & 60.9  & .19 & .26 & 74.9\\
    & DM &  .31 & .38 & 7.93 & .32 & .36 & 18.5 & .27 & .29 & 48.2 & .25 & .29 & 116\\
    &  max s.e. & .01 & .01 & - & .01 & .01 & - & .01 & .01 & - & .00 & .00 & - \\ \hline
     \end{tabular*}
  \end{center}
\end{table}

\begin{figure}[!htb]
\minipage{0.49\textwidth}
  \includegraphics[width=\linewidth]{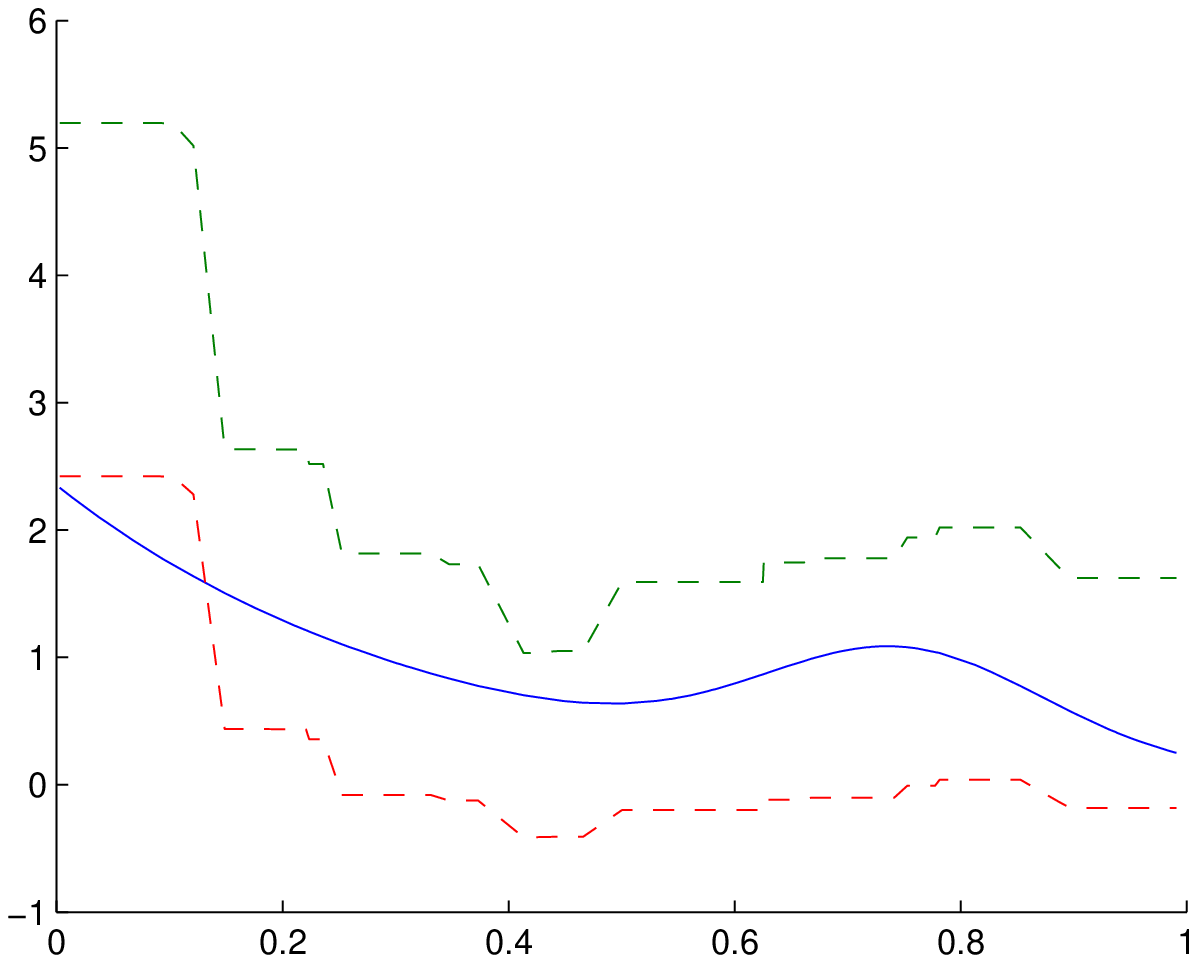}
  \endminipage\hfill
\minipage{0.49\textwidth}
  \includegraphics[width=\linewidth]{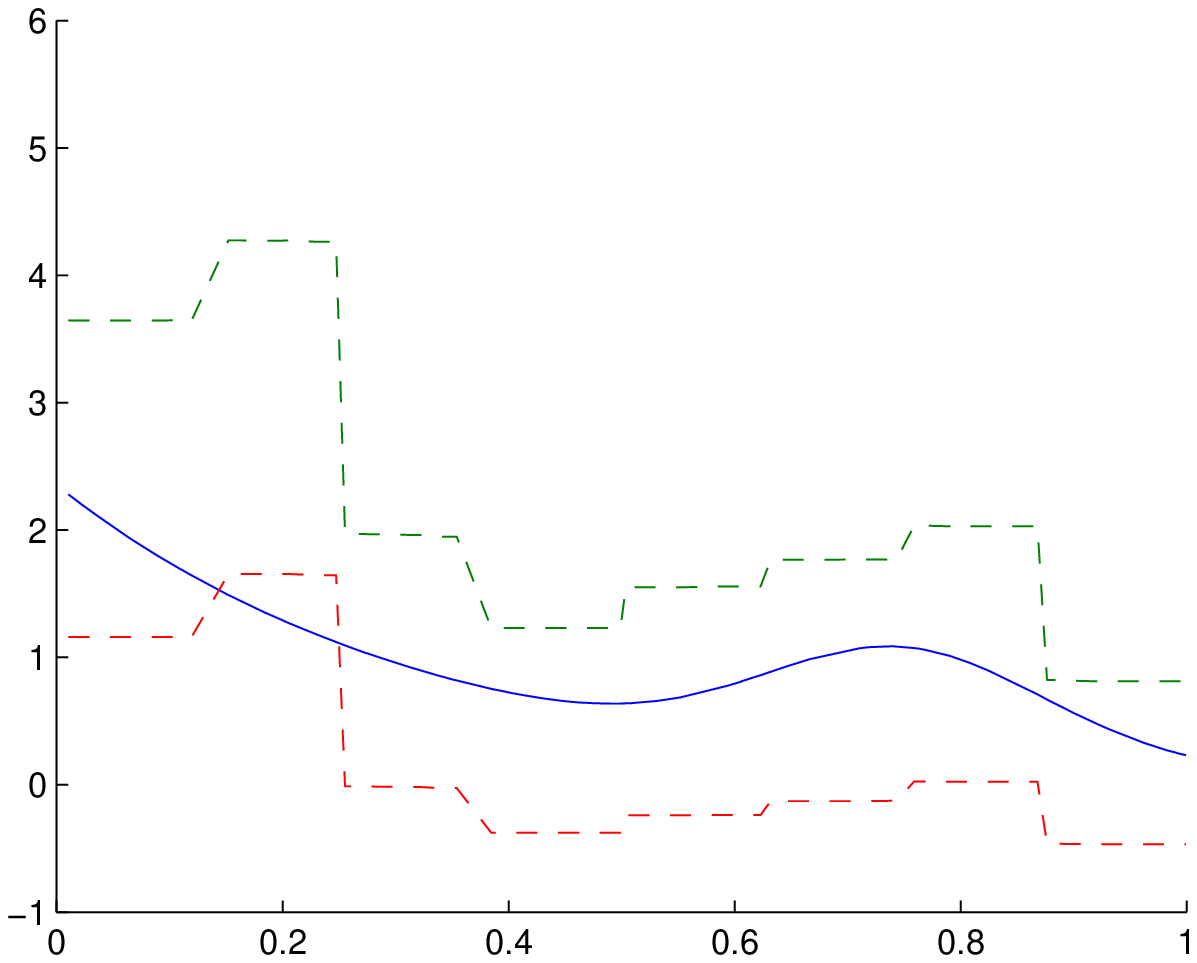}
\endminipage\hfill \\
\minipage{0.49\textwidth}
  \includegraphics[width=\linewidth]{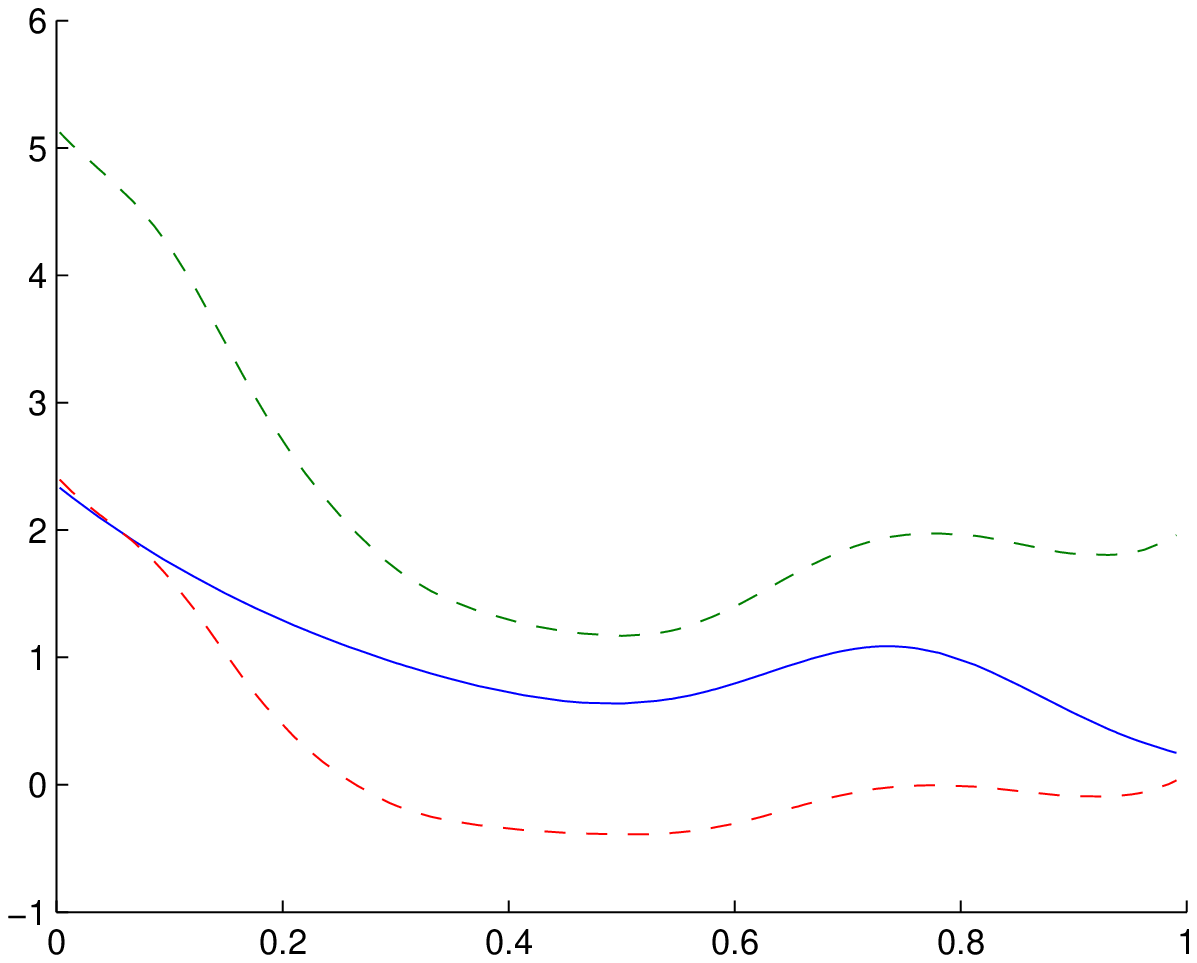}
  \endminipage\hfill
\minipage{0.49\textwidth}
  \includegraphics[width=\linewidth]{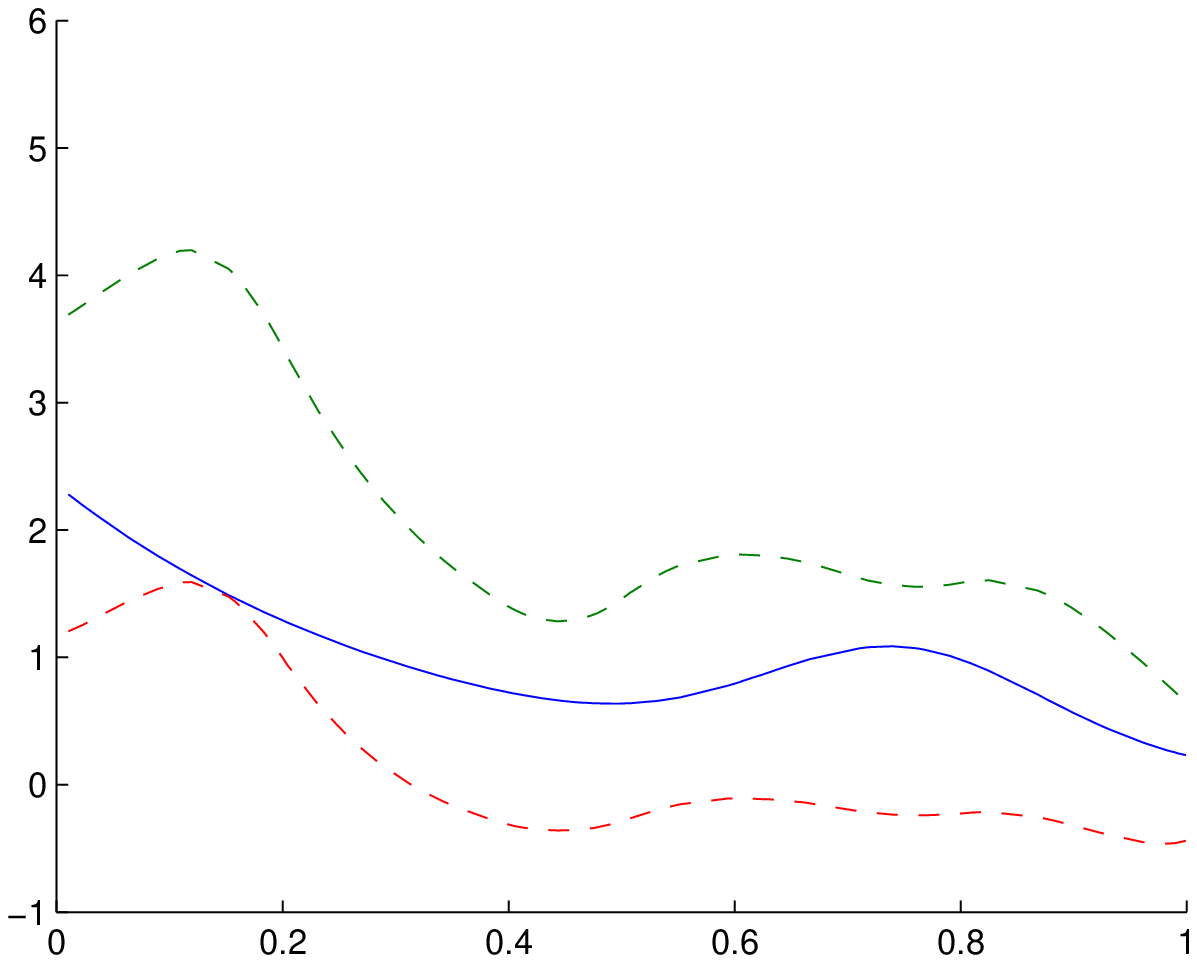}
\endminipage\hfill
\caption{Confidence bands for density estimation example \eqref{mixden}. Dashed lines: confidence bands; solid line: true density. Upper left: $q=1$, $n=100$; upper right: $q=1$, $n=500$; lower left: $q=3$, $n=100$; lower right: $q=3$, $n=500$.}\label{fig:den}
\end{figure}

\subsection{Real data example}
Next, we present a real data analysis of a functional linear model.
\begin{example} \rm\label{exreal1} ~  The Tecator data (http://lib.stat.cmu.edu/datasets/tecator) provides an example of functional data where the spectra of meat samples are observed. The objective is to identify important chemical components and predict the fat content. The data consists of $172$ training and $43$ testing samples, where each sample contains $100$ channel spectrum of absorbents.

We consider a functional linear model in \eqref{funcscalar} and use a B-spline basis expansion of $\beta(t)$. A prior is assigned by putting a Zellner's g-prior on the coefficients, a geometric distribution on $J$ truncated between $5$ and $15$ and an inverse gamma distribution $\text{IG}(a,b)$ on $\sigma^2$. We let the values of hyperparameters $g,a,b$ range from $1$ to $100$ and the posterior results are quite insensitive. The MCMC-free calculation yields a root mean squared error (RMSE) of prediction $2.64$ for $q=1$ and RMSE$=2.49$ for $q=3$, which are generally better than the regression model results (RMSE $\geq 4$) built based on principal component analysis.
\end{example}

\appendix

\section*{Appendix: B-splines}
Here we provide a brief introduction to B-splines; more details are given in \citet{Deboor2001}. Let the unit interval $[0,1]$ be divided into $K$ equally spaced subintervals. Splines are continuous, piecewise polynomials of degree at most $q$, $(q-2)$ times continuously differentiable and form a $J=q+K-1$ dimensional linear space. B-splines provide a convenient basis for this space. B-splines are always nonnegative, add up to one and each basis function is supported on an interval of length at most $q/K$.

Define the scaled B-spline basis functions $B_j^* = B_j/\int_0^1 B_j$, $j=1,\ldots,J$, so that $\int_0^1 B_j^*(z)dz=1$, $j=1,\ldots,J$.
Denote the column vector of B-spline basis functions by $\bm{B}$ and that of the normalized B-spline basis functions by $\bm{B}^*$.
The following results show some useful approximation properties of (tensor-product) B-splines.

\begin{lemma}\label{lemma:l00}
\begin{itemize}
\item [(a)] For any function $f \in \mathcal{C}^{\alpha}(0,1)$, $0 < \alpha \leq q$, there exists $\bm{\theta}\in \mathbb{R}^J$ and a constant $C > 0$ that depends only on $q$ such that
$
\|f - \bm{\theta}^T \bm{B} \|_{\infty} \leq C J^{-\alpha} \|f^{(\alpha)}\|_{\infty}.
$
\item [(b)] Further, if $f > 0$ we can choose every element of $\bm{\theta}$ to be positive.
\item [(c)] If $0 < f < 1$, we can choose every element of $\bm{\theta}$ to be between $0$ and $1$.
\item [(d)] Define $B_j^* = B_j/\int_0^1 B_j(z)dz$ for $j=1,\ldots,J$, and $\bm{B}^*$ as the column vector $(B_1^*,\ldots,B_J^*)$. If $f$ is a density function, then there exists $\bm{\theta}\in \Delta_J$ and a constant $C > 0$ such that
$
\|f - \bm{\theta}^T \bm{B}^* \|_{\infty} \leq C J^{-\alpha} \|f^{(\alpha)}\|_{\infty}.
$
\end{itemize}
\end{lemma}
\begin{remark}\rm
In part (b), the condition $f > 0$ is crucial. If we approximate a nonnegative function $f$ using nonnegative coefficients $\bm{\theta}$, then the approximation error is only $O(J^{-1})$ [cf. \citep{Deboor1974}], which does not adapt to smoothness levels beyond $1$.
\end{remark}

\begin{proof}[Proof of Lemma 1]
The first part is a well-known spline approximation result, e.g., Theorem 6.10 in \citet{Schumaker2007}.

For the second assertion, find $\epsilon >0$ such that $f \geq \epsilon$. Using Corollaries 4 and 6 in Chapter 11 of \citet{Deboor2001}, for each $\theta_j$, there exists a universal constant $C_1$ that depends only on $q$, such that $|\theta_j - c| \leq C_1 \sup_{x \in [t_{j+1},t_{j+q-1}]} |f(z)-c |$ for any choice of the constant $c$; here $t_{j+1}$ and $t_{j+q-1}$ are $(j+1)$th and $(j+q-1)$th knots. Choose $c = \inf_{z \in [t_{j+1},t_{j+q-1}]} f(z) \geq \epsilon$, and note that the infimum is attained somewhere in $[t_{j+1},t_{j+q-1}]$, say at $t^*$. By the smoothness condition on $f$, we have $\sup_{z \in [t_{j+1},t_{j+q-1}]} |f(z)-c | \leq C_2 |z-t^*|^{\min(\alpha,1)} \leq C_2 (q/J)^{\min(\alpha,1)}$ for some constant $C_2 > 0$. Choosing $J > q (C_1 C_2/\epsilon)^{\max(1/\alpha,1)}$, we have $\theta_j > c - C_1 (q/J)^{\min(\alpha,1)} \geq 0$.

Part (c) is a consequence of (b) by considering $1-f > 0$.

For part (d), by (b), we know there exists a $\bm{\eta}_1 \in (0,\infty)^J$ such that $\|f - \bm{\eta}_1^T \bm{B}\| \lesssim J^{-\alpha} $. Define $\eta_{2,i}=\eta_{1,j} \int_0^1 B_j(z)dz$ for $j=1,\ldots, J$. Then $\|f - \bm{\eta}_2^T \bm{B^*}\|_{\infty} \lesssim J^{-\alpha} $, and in particular $\|\bm{\eta}_2^T \bm{B}\|$ is bounded. By integration, we have $|1-\|\bm{\eta}_2\|_1|=|1 - \sum_{j=1}^J \eta_{2,j}| \lesssim J^{-\alpha}$. Choose $\bm{\theta} = \bm{\eta}_2/\|\bm{\eta}_2\|_1 \in \Delta_J$. Note that $\|f - \bm{\theta}^T \bm{B}^*\|_{\infty} \leq \|f - \bm{\eta}_2^T \bm{B}^*\|_{\infty} + \|\bm{\eta}_2^T \bm{B}^*\|_{\infty}|1-\|\bm{\eta}_2\|_1^{-1}|  \lesssim J^{-\alpha}$.
\end{proof}

\begin{lemma}\label{lemma:l00*}
Let $B_{j_1\cdots j_s}(z_1,\ldots,z_s)=\prod_{k=1}^s B_{j_k}(z_k)$, $1\le j_k\le J_k$, $k=1,\ldots,s$, be the tensor products of B-splines formed by univariate B-splines in $J_k$-dimensional space of splines, $k=1,\ldots,s$ respectively.
\begin{itemize}
\item [(a)] For any function $f \in \mathcal{C}^{\bm{\alpha}}(0,1)^s$, the anisotropic H\"older class defined in Section~3.2, where $\alpha_1,\ldots,\alpha_s$ are positive integers less than or equal to $q$, there exists $\bm{\theta}=(\theta_{j_1\cdots j_s}: 1\le j_k\le J_k, k=1,\ldots,s)\in \mathbb{R}^{\prod_{k=1}^s J_k}$ and a constant $C > 0$ that depends only on $q$ such that
$$
\|f - \bm{\theta}^T \bm{B} \|_{\infty} \leq C \sum_{k=1}^s J_k^{-\alpha_k} \left\|\frac {\partial^{\alpha_k}f}{\partial z_k^{\alpha_k}}\right\|_{\infty}.
$$
\item [(b)] Further, if $f > 0$ we can choose every component of $\bm{\theta}$ to be positive.
\item [(c)] If $0 < f < 1$, we can choose every element of $\bm{\theta}$ to be between $0$ and $1$.
\item [(d)] Define $B_{j_1\cdots j_s}^* = B_{j_1\cdots j_s}/\int_0^1 B_{j_1\cdots j_s}(z_1,\ldots,z_s)dz_1\cdots dz_s$, for $j_k=1,\ldots,J_k$, $k=1,\ldots,s$, and $\bm{B}^*$ as the column vector formed by the collection $B_{j_1\cdots j_s}^*$. If $f$ is a density function, then there exists $\bm{\theta}\in \Delta_{\prod_{k=1}^s J_k}$ and a constant $C > 0$ such that the same approximation order is maintained.
\end{itemize}
\end{lemma}

\begin{proof}
The first assertion is established in Theorem 12.7 in \citet{Schumaker2007}.

Proof of the second assertion proceeds as in the corresponding part of Lemma 1 using the parallel properties of tensor products of B-splines. The only relation we need to verify is $|\theta_{j_1\cdots j_s} - c| \leq C_1 \max_k \sup_{z_k \in [t_{i+1,k},t_{i+q-1,k}]} |f(z)-c |$ for any choice of the constant $c$; here $t_{i+1,k}$ and $t_{i+q-1,k}$ are $(i+1)$th and $(i+q-1)$th knots on the $k$th co-ordinate, $k=1,\ldots,s$. As in the univariate case, because the sum of all multivariate B-splines is one, to establish the relation we need to bound absolute values of the coefficients using the values of the target function. Clearly a dual basis for the multivariate B-splines is formed by tensor products of univariate dual bases and these can be chosen to be uniformly bounded; see Theorem~4.41 of \citet{Schumaker2007}. Using such a dual basis, the maximum value of coefficients of spline approximations is bounded by a constant multiple of the $L_\infty$-norm of the target function. This gives the desired bound.

Parts (c) and (d) are established following exactly the same arguments used in the respective parts in the Proof of Lemma 1.
\end{proof}

\begin{remark}\rm
In the isotropic case $\bm{\alpha}=(\alpha,\ldots,\alpha)$, the value of $\alpha$ need not be restricted to integers only --- any $\alpha\le q$ can be treated. This is because in this case the approximation error $\|f-\bm{\theta}^T \bm{B}\|_\infty$ for the best multivariate spline approximation for $f\in \mathcal{C}^\alpha (0,1)^s$ with $J$ terms in each direction decays at the rate $J^{-\alpha}$ for any positive $\alpha\le q$.
\end{remark}

\begin{remark}\rm
In part (b), the condition $f > 0$ is crucial. If we approximate a nonnegative function $f$ using nonnegative coefficients $\bm{\theta}$, then the approximation error is only $O(J^{-1})$ \citep{Deboor1974}, which does not adapt to smoothness levels beyond $1$.
\end{remark}

\bibliographystyle{plainnat}
\bibliography{ref}

\begin{thebibliography}{52}
\providecommand{\natexlab}[1]{#1}
\providecommand{\url}[1]{\texttt{#1}}
\expandafter\ifx\csname urlstyle\endcsname\relax
  \providecommand{\doi}[1]{doi: #1}\else
  \providecommand{\doi}{doi: \begingroup \urlstyle{rm}\Url}\fi

\bibitem[Arbel et~al.(2013)Arbel, Gayraud, and Rousseau]{Arbel2013}
J.~Arbel, G.~Gayraud, and J.~Rousseau.
\newblock Bayesian optimal adaptive estimation using a sieve prior.
\newblock \emph{Scandinavian Journal of Statistics}, 40:\penalty0 549--570,
  2013.

\bibitem[Asmussen and Glynn(2007)]{Asm2007}
S.~Asmussen and P.~W. Glynn.
\newblock \emph{Stochastic Simulation: Algorithms and Analysis}.
\newblock Springer, 2007.

\bibitem[Babenko and Belitser(2010)]{Babenko2010}
A.~Babenko and E.~Belitser.
\newblock Oracle convergence rate of posterior under projection prior and
  {B}ayesian model selection.
\newblock \emph{Mathematical Methods of Statistics}, 19:\penalty0 219--245,
  2010.

\bibitem[Banerjee et~al.(2008)Banerjee, Gelfand, Finley, and
  Sang]{Banerjee2008}
S.~Banerjee, A.~E. Gelfand, A.~O. Finley, and H.~Sang.
\newblock Gaussian predictive process models for large spatial data sets.
\newblock \emph{Journal of the Royal Statistical Society, Series B},
  70:\penalty0 825--848, 2008.

\bibitem[Belitser and Ghosal(2003)]{Belitser2003}
E.~Belitser and S.~Ghosal.
\newblock Adaptive {B}ayesian inference on the mean of an infinite-dimensional
  normal distribution.
\newblock \emph{The Annals of Statistics}, 31:\penalty0 536--559, 2003.

\bibitem[Bhattacharya et~al.(2014)Bhattacharya, Pati, and Dunson]{Bha2014}
A.~Bhattacharya, D.~Pati, and D.B. Dunson.
\newblock Anisotropic function estimation with multi-bandwidth gaussian
  process.
\newblock \emph{The Annals of Statistics}, 32:\penalty0 352--381, 2014.

\bibitem[Cai and Yuan(2011)]{Cai2011}
T~Cai and M.~Yuan.
\newblock Optimal estimation of the mean function based on discretely sampled
  funcitonal data: phase transition.
\newblock \emph{The Annals of Statistics}, 39:\penalty0 2330--2355, 2011.

\bibitem[Cardot et~al.(2003)Cardot, Ferraty, and Sarda]{Cardot2003}
H.~Cardot, F.~Ferraty, and P.~Sarda.
\newblock Spline estimators for the functional linear model.
\newblock \emph{Statistica Sinica}, 13:\penalty0 571--591, 2003.

\bibitem[Castillo(2008)]{Castillo2008}
I.~Castillo.
\newblock Lower bounds for posterior rates with {G}aussian process priors.
\newblock \emph{Electronic Journal of Statistics}, 2:\penalty0 1281--1299,
  2008.

\bibitem[Castillo(2012)]{Castillo2012}
I.~Castillo.
\newblock A semi-parametric {B}ernstein-von {M}ises theorem for {G}aussian
  process priors.
\newblock \emph{Probability Theory and Related Fields}, 152:\penalty0 53--99,
  2012.

\bibitem[Castillo et~al.(2014)Castillo, Kerkyacharian, and
  Picard]{Castillo2014}
I.~Castillo, G.~Kerkyacharian, and D.~Picard.
\newblock Thomas bayes' walk on manifolds.
\newblock \emph{Probability Theory and Related Fields}, 158:\penalty0 665--710,
  2014.

\bibitem[Choi and Schervish(2007)]{Choi2007}
T.~Choi and M.~J. Schervish.
\newblock On posterior consistency in nonparametric regression problems.
\newblock \emph{Journal of Multivariate Analysis}, 98:\penalty0 1969--1987,
  2007.

\bibitem[Choudhuri et~al.(2007)Choudhuri, Ghosal, and Roy]{Choudhuri2007}
N~Choudhuri, S.~Ghosal, and A.~Roy.
\newblock Nonparametric binary regression using a {G}aussian process prior.
\newblock \emph{Statistical Methodology}, 4:\penalty0 227--243, 2007.

\bibitem[Cohen et~al.(1993)Cohen, Daubechies, and Vial]{Cohen1993}
A.~Cohen, I.~Daubechies, and P.~Vial.
\newblock Wavelets on the interval and fast wavelet transforms.
\newblock \emph{Applied and Computational Harmonic Analysis}, 1:\penalty0
  54--81, 1993.

\bibitem[Crainiceanu et~al.(2005)Crainiceanu, Ruppert, and
  Wand]{Crainiceanu2005}
C.~M. Crainiceanu, R.~Ruppert, and M.~P. Wand.
\newblock Bayesian analysis for penalized spline regression using winbugs.
\newblock \emph{Journal of Statistical Software}, 14:\penalty0 1--24, 2005.

\bibitem[Dai and Xu(2013)]{Dai2013}
F.~Dai and Y.~Xu.
\newblock \emph{Approximation theory and harmonic analysis on spheres and
  balls}.
\newblock Springer Monographs in Mathematics, 2013.

\bibitem[de~Boor and Daniel(1974)]{Deboor1974}
C.~de~Boor and J.~W. Daniel.
\newblock Splines with nonnegative b-spline coefficients.
\newblock \emph{Mathematics of Computation}, 28:\penalty0 565--568, 1974.

\bibitem[de~Boor(2001)]{Deboor2001}
Carl de~Boor.
\newblock \emph{A {P}ractical {G}uide to {S}plines}.
\newblock Springer, 2001.

\bibitem[de~Jonge and van Zanten(2012)]{dejonge2012}
R.~de~Jonge and H.~van Zanten.
\newblock Adaptive estimation of multivariate functions using conditionally
  {G}aussian tensor-product spline priors.
\newblock \emph{Electronic Journal of Statistics}, 6:\penalty0 1984--2001,
  2012.

\bibitem[Escobar and West(1995)]{Escobar1995}
M.~D. Escobar and M.~West.
\newblock Bayesian density estimation and inference using mixtures.
\newblock \emph{Journal of the American Statistical Association}, 90:\penalty0
  577--588, 1995.

\bibitem[Gao and Zhou(2013)]{Gao2013}
C.~Gao and H.~H. Zhou.
\newblock Adaptive bayesian estimation via block prior.
\newblock Technical report, arXiv:1312.3937, 2013.

\bibitem[Ghosal(2001)]{Ghosal20012}
S.~Ghosal.
\newblock Convergence rates for density estimation with {B}ernstein
  polynomials.
\newblock \emph{The Annals of Statistics}, 29\penalty0 (5):\penalty0
  1264--1280, 2001.

\bibitem[Ghosal and Roy(2006)]{Ghosal2006}
S.~Ghosal and A.~Roy.
\newblock Posterior consistency of {G}aussian process prior for nonparametric
  binary regression.
\newblock \emph{The Annals of Statistics}, 34:\penalty0 2413--2429, 2006.

\bibitem[Ghosal and van~der Vaart(2007{\natexlab{a}})]{Ghosal2007}
S.~Ghosal and A.~van~der Vaart.
\newblock Posterior convergence rates of {D}irichlet mixtures at smooth
  densities.
\newblock \emph{The Annals of Statistics}, 35\penalty0 (3):\penalty0 697--723,
  2007{\natexlab{a}}.

\bibitem[Ghosal and van~der Vaart(2007{\natexlab{b}})]{Ghosal20071}
S.~Ghosal and A.~van~der Vaart.
\newblock Convergence rates of posterior distributions for noniid observations.
\newblock \emph{The Annals of Statistics}, 35:\penalty0 192--223,
  2007{\natexlab{b}}.

\bibitem[Ghosal et~al.(2000)Ghosal, Ghosh, and van~der Vaart]{Ghosal2000}
S.~Ghosal, J.~K. Ghosh, and A.~van~der Vaart.
\newblock Convergence rates of posterior distributions.
\newblock \emph{The Annals of Statistics}, 28\penalty0 (2):\penalty0 500--531,
  2000.

\bibitem[Ghosal et~al.(2003)Ghosal, Lember, and van~der Vaart]{Ghosal2002}
S.~Ghosal, J.~Lember, and A.~van~der Vaart.
\newblock On {B}ayesian adaptation.
\newblock \emph{In {P}roceedings of the {E}ighth {V}ilnius {C}onference on
  {P}robability {T}heory and {M}athematical {S}tatistics, Part II (2002)},
  79:\penalty0 165--175, 2003.

\bibitem[Ghosal et~al.(2008)Ghosal, Lember, and van~der Vaart]{Ghosal2008}
S.~Ghosal, J.~Lember, and A.~van~der Vaart.
\newblock Nonparametric {B}ayesian model selection and averaging.
\newblock \emph{Electronic Journal of Statistics}, 2:\penalty0 63--89, 2008.

\bibitem[Goldsmith et~al.(2011)Goldsmith, Wand, and Crainiceanu]{Goldsmith2011}
J~Goldsmith, Matt~P. Wand, and Ciprian Crainiceanu.
\newblock Functional regression via variational {B}ayes.
\newblock \emph{Electronic Journal of Statistics}, 5:\penalty0 572--602, 2011.

\bibitem[Hall and Horowitz(2007)]{Hall2007}
Peter Hall and Joel~L. Horowitz.
\newblock Methodology and convergence rates for functional linear regression.
\newblock \emph{The Annals of Statistics}, 35:\penalty0 70--91, 2007.

\bibitem[Hasminskii(1978)]{Hasminskii1978}
R.~Z. Hasminskii.
\newblock A lower bound on the risks of nonparametric estimates of densities in
  the uniform metric.
\newblock \emph{Theory of Probability and Its Applications}, 23:\penalty0
  794--796, 1978.

\bibitem[Hesthaven et~al.(2007)Hesthaven, Gottlieb, and
  Gottlieb]{Hesthaven2007}
J.~S. Hesthaven, S.~Gottlieb, and D.~Gottlieb.
\newblock \emph{Spectral {M}ethods for {T}ime-{D}ependent {P}roblems}.
\newblock Cambridge University Press, 2007.

\bibitem[Huang(2004)]{Huang2004}
T.-M. Huang.
\newblock Convergence rates for posterior distributions and adaptive
  estimation.
\newblock \emph{The Annals of Statistics}, 32:\penalty0 1556--1593, 2004.

\bibitem[Kruijer and van~der Vaart(2008)]{Kruijer2008}
W.~Kruijer and A.~van~der Vaart.
\newblock Posterior convergence rates for dirichlet mixtures of beta densities.
\newblock \emph{Journal of Statistical Planning and Inference}, 138:\penalty0
  1981--1992, 2008.

\bibitem[Lenk(1988)]{Lenk1988}
P.~J. Lenk.
\newblock The logistic normal distribution for {B}ayesian, nonparametric,
  predictive densities.
\newblock \emph{Journal of the American Statistical Association}, 83:\penalty0
  509--516, 1988.

\bibitem[Lian(2011)]{Lian2011}
H.~Lian.
\newblock On posterior distribution of {B}ayesian wavelet thresholding.
\newblock \emph{Journal of Statistical Planning and Inference}, 141:\penalty0
  318--324, 2011.

\bibitem[Lorenz(1953)]{Lorenz1953}
G.~G. Lorenz.
\newblock \emph{Bernstein {P}olynomials}.
\newblock Univ. Toronto Press, 1953.

\bibitem[Petrone(1999)]{Petrone1999b}
S.~Petrone.
\newblock Bayesian density estimation using {B}ernstein polynomials.
\newblock \emph{Canadian Journal of Statistics}, 27:\penalty0 105--126, 1999.

\bibitem[Rasmussen and Williams(2006)]{Rasmussen2006}
C.~E. Rasmussen and C.~K. Williams.
\newblock \emph{Gaussian {P}rocesses for {M}achine {L}earning}.
\newblock MIT Press, 2006.

\bibitem[Rivoirard and Rousseau(2012{\natexlab{a}})]{Rivoirard2012}
V.~Rivoirard and J.~Rousseau.
\newblock Posterior concentration rates for infinite dimensional exponential
  families.
\newblock \emph{Bayesian Analysis}, 7:\penalty0 311--334, 2012{\natexlab{a}}.

\bibitem[Rivoirard and Rousseau(2012{\natexlab{b}})]{Rivoirard20121}
V.~Rivoirard and J.~Rousseau.
\newblock Bernstein-von mises theorem for linear functionals of the density.
\newblock \emph{The Annals of Statistics}, 40:\penalty0 1489--1523,
  2012{\natexlab{b}}.

\bibitem[Rue et~al.(2009)Rue, Martino, and Chopin]{Rue2009}
H.~Rue, S.~Martino, and N.~Chopin.
\newblock Approximate {B}ayesian inference for latent {G}aussian models by
  using integrated nested {L}aplace approximations.
\newblock \emph{Journal of the Royal Statistical Society, Series B},
  71:\penalty0 319--392, 2009.

\bibitem[Schumaker(2007)]{Schumaker2007}
L.~Schumaker.
\newblock \emph{Spline Functions: Basic Theory}.
\newblock Cambridge University Press, 2007.

\bibitem[Scricciolo(2006)]{Scricciolo2006}
C.~Scricciolo.
\newblock Convergence rates for {B}ayesian density estimation on
  infinite-dimensional exponential families.
\newblock \emph{The Annals of Statistics}, 34:\penalty0 2897--2920, 2006.

\bibitem[Shen et~al.(2013)Shen, Tokdar, and Ghosal]{Shen2013}
W.~Shen, S.~T. Tokdar, and S.~Ghosal.
\newblock Adaptive bayesian multivariate density estimation with dirichlet
  mixtures.
\newblock \emph{Biometrika}, 100:\penalty0 623--640, 2013.

\bibitem[Shen and Wasserman(2001)]{Shen2001}
X.~Shen and L.~Wasserman.
\newblock Rates of convergence of posterior distributions.
\newblock \emph{The Annals of Statistics}, 29:\penalty0 687--714, 2001.

\bibitem[Szab\'o et~al.(2013)Szab\'o, van~der Vaart, and van Zanten]{szabo2013}
B.~T. Szab\'o, A.~W. van~der Vaart, and J.~H. van Zanten.
\newblock Empirical bayes scaling of gaussian priors in the white noise model.
\newblock \emph{Electronic Journal of Statistics}, 7:\penalty0 991--1018, 2013.
\newblock \doi{10.1214/13-EJS798}.
\newblock URL \url{http://dx.doi.org/10.1214/13-EJS798}.

\bibitem[Tokdar(2007)]{Tokdar2007}
S.~T. Tokdar.
\newblock Towards a faster implementation of density estimation with logistic
  gaussian process priors.
\newblock \emph{Journal of Computational and Graphical Statistics},
  16:\penalty0 633--655, 2007.

\bibitem[Tokdar and Ghosh(2007)]{Tokdar20071}
S.~T. Tokdar and J.~K. Ghosh.
\newblock Posterior consistency of logistic gaussian process priors in density
  estimation.
\newblock \emph{Journal of Statistical Planning and Inference}, 137:\penalty0
  34--42, 2007.

\bibitem[van~der Vaart and van Zanten(2007)]{Vander2007}
A.~van~der Vaart and H.~van Zanten.
\newblock Bayesian inference with rescaled {G}aussian process priors.
\newblock \emph{Electronic Journal of Statistics}, 1:\penalty0 433--448, 2007.

\bibitem[van~der Vaart and van Zanten(2008)]{Vander2008}
A.~van~der Vaart and H.~van Zanten.
\newblock Rates of contraction of posterior distributions based on {G}aussian
  process priors.
\newblock \emph{The Annals of Statistics}, 36:\penalty0 1435--1463, 2008.

\bibitem[van~der Vaart and van Zanten(2009)]{Vander2009}
A.~van~der Vaart and H.~van Zanten.
\newblock Adaptive {B}ayesian estimation using a {G}aussian random field with
  inverse gamma bandwidth.
\newblock \emph{The Annals of Statistics}, 37\penalty0 (5B):\penalty0
  2655--2675, 2009.

\end{thebibliography}

\end{document}